\documentclass[reqno,draft]{amsart}
\usepackage{amssymb,amscd,url}

\begin{document}

\allowdisplaybreaks

\newcommand{\DATE}{\today}
\newcommand{\TITLE}{Ramification Filtrations of Certain Abelian Lie Extensions of Local Fields}
\newcommand{\TITLERUNNING}{Abelian Lie Extensions}


\newtheorem{theorem}{Theorem}[section]
\newtheorem*{Theorem}{Theorem}
\newtheorem{lemma}[theorem]{Lemma}
\newtheorem{conjecture}[theorem]{Conjecture}
\newtheorem{proposition}[theorem]{Proposition}
\newtheorem{corollary}[theorem]{Corollary}
\newtheorem*{claim}{Claim}

\theoremstyle{definition}
\newtheorem{question}{Question}
\renewcommand{\thequestion}{\Alph{question}}
\newtheorem*{definition}{Definition}
\newtheorem{example}[theorem]{Example}

\theoremstyle{remark}
\newtheorem{remark}{Remark}
\newtheorem*{acknowledgement}{Acknowledgements}



\newenvironment{notation}[0]{%
  \begin{list}%
    {}%
    {\setlength{\itemindent}{0pt}
     \setlength{\labelwidth}{4\parindent}
     \setlength{\labelsep}{\parindent}
     \setlength{\leftmargin}{5\parindent}
     \setlength{\itemsep}{0pt}
     }%
   }%
  {\end{list}}

\newenvironment{parts}[0]{%
  \begin{list}{}%
    {\setlength{\itemindent}{0pt}
     \setlength{\labelwidth}{1.5\parindent}
     \setlength{\labelsep}{.5\parindent}
     \setlength{\leftmargin}{2\parindent}
     \setlength{\itemsep}{0pt}
     }%
   }%
  {\end{list}}
\newcommand{\Part}[1]{\item[\upshape#1]}

\renewcommand{\a}{\alpha}
\renewcommand{\b}{\beta}
\renewcommand{\c}{\chi}
\renewcommand{\d}{\delta}
\newcommand{\e}{\epsilon}
\newcommand{\ve}{\varepsilon}
\newcommand{\g}{\gamma}
\renewcommand{\i}{\iota}
\renewcommand{\k}{\kappa}
\renewcommand{\l}{\lambda}
\renewcommand{\L}{\Lambda}
\newcommand{\m}{\mu}
\newcommand{\n}{\nu}
\newcommand{\om}{\omega}
\newcommand{\p}{\phi}
\newcommand{\ps}{\psi}
\newcommand{\vp}{\varepsilon}
\renewcommand{\r}{\rho}
\newcommand{\s}{\sigma}
\newcommand{\vs}{\varsigma}
\renewcommand{\t}{\tau}

\renewcommand{\th}{\theta}
\newcommand{\vth}{\vartheta}
\renewcommand{\u}{\upsilon}
\newcommand{\x}{\xi}
\newcommand{\z}{\zeta}

\newcommand{\ga}{{\mathfrak{a}}}
\newcommand{\gA}{{\mathfrak{A}}}
\newcommand{\gb}{{\mathfrak{b}}}
\newcommand{\gc}{{\mathfrak{c}}}
\newcommand{\gd}{{\mathfrak{d}}}
\newcommand{\efk}{{\mathfrak{e}}}
\newcommand{\gf}{{\mathfrak{f}}}
\newcommand{\gfk}{{\mathfrak{g}}}
\newcommand{\gh}{{\mathfrak{h}}}
\newcommand{\gi}{{\mathfrak{i}}}
\newcommand{\gj}{{\mathfrak{j}}}
\newcommand{\gm}{{\mathfrak{m}}}
\newcommand{\gn}{{\mathfrak{n}}}
\newcommand{\go}{{\mathfrak{o}}}
\newcommand{\gO}{{\mathfrak{O}}}
\newcommand{\gp}{{\mathfrak{p}}}
\newcommand{\gP}{{\mathfrak{P}}}
\newcommand{\gq}{{\mathfrak{q}}}
\newcommand{\gr}{{\mathfrak{r}}}

\newcommand{\gR}{{\mathfrak{R}}}

\newcommand{\Abar}{{\bar A}}
\newcommand{\Ebar}{{\bar E}}
\newcommand{\Kbar}{{\bar K}}
\newcommand{\Pbar}{{\bar P}}
\newcommand{\Sbar}{{\bar S}}
\newcommand{\Tbar}{{\bar T}}
\newcommand{\ybar}{{\bar y}}
\newcommand{\phibar}{{\bar\f}}

\newcommand{\Acal}{{\mathcal A}}
\newcommand{\Bcal}{{\mathcal B}}
\newcommand{\Ccal}{{\mathcal C}}
\newcommand{\Dcal}{{\mathcal D}}
\newcommand{\Ecal}{{\mathcal E}}
\newcommand{\Fcal}{{\mathcal F}}
\newcommand{\Gcal}{{\mathcal G}}
\newcommand{\Hcal}{{\mathcal H}}
\newcommand{\Ical}{{\mathcal I}}
\newcommand{\Jcal}{{\mathcal J}}
\newcommand{\Kcal}{{\mathcal K}}
\newcommand{\Lcal}{{\mathcal L}}
\newcommand{\Mcal}{{\mathcal M}}
\newcommand{\Ncal}{{\mathcal N}}
\newcommand{\Ocal}{{\mathcal O}}
\newcommand{\Pcal}{{\mathcal P}}
\newcommand{\Qcal}{{\mathcal Q}}
\newcommand{\Rcal}{{\mathcal R}}
\newcommand{\Scal}{{\mathcal S}}
\newcommand{\Tcal}{{\mathcal T}}
\newcommand{\Ucal}{{\mathcal U}}
\newcommand{\Vcal}{{\mathcal V}}
\newcommand{\Wcal}{{\mathcal W}}
\newcommand{\Xcal}{{\mathcal X}}
\newcommand{\Ycal}{{\mathcal Y}}
\newcommand{\Zcal}{{\mathcal Z}}

\renewcommand{\AA}{\mathbb{A}}
\newcommand{\BB}{{\mathbb B}}
\newcommand{\CC}{{\mathbb C}}
\newcommand{\DD}{{\mathbb D}}
\newcommand{\EE}{{\mathbb E}}
\newcommand{\FF}{{\mathbb F}}
\newcommand{\GG}{{\mathbb G}}
\newcommand{\HH}{{\mathbb H}}
\newcommand{\II}{{\mathbb I}}
\newcommand{\JJ}{{\mathbb J}}
\newcommand{\KK}{{\mathbb K}}
\newcommand{\LL}{{\mathbb L}}
\newcommand{\MM}{{\mathbb M}}
\newcommand{\NN}{{\mathbb N}}
\newcommand{\OO}{{\mathbb O}}
\newcommand{\PP}{{\mathbb P}}
\newcommand{\QQ}{{\mathbb Q}}
\newcommand{\RR}{{\mathbb R}}
\renewcommand{\SS}{{\mathbb S}}
\newcommand{\TT}{{\mathbb T}}
\newcommand{\UU}{{\mathbb U}}
\newcommand{\VV}{{\mathbb V}}
\newcommand{\WW}{{\mathbb W}}
\newcommand{\XX}{{\mathbb X}}
\newcommand{\YY}{{\mathbb Y}}
\newcommand{\ZZ}{{\mathbb Z}}

\newcommand{\bfa}{{\mathbf a}}
\newcommand{\bfb}{{\mathbf b}}
\newcommand{\bfc}{{\mathbf c}}
\newcommand{\bfe}{{\mathbf e}}
\newcommand{\bff}{{\mathbf f}}
\newcommand{\bfg}{{\mathbf g}}
\newcommand{\bfp}{{\mathbf p}}
\newcommand{\bfr}{{\mathbf r}}
\newcommand{\bfs}{{\mathbf s}}
\newcommand{\bft}{{\mathbf t}}
\newcommand{\bfu}{{\mathbf u}}
\newcommand{\bfv}{{\mathbf v}}
\newcommand{\bfw}{{\mathbf w}}
\newcommand{\bfx}{{\mathbf x}}
\newcommand{\bfy}{{\mathbf y}}
\newcommand{\bfz}{{\mathbf z}}
\newcommand{\bfA}{{\mathbf A}}
\newcommand{\bfF}{{\mathbf F}}
\newcommand{\bfB}{{\mathbf B}}
\newcommand{\bfD}{{\mathbf D}}
\newcommand{\bfG}{{\mathbf G}}
\newcommand{\bfI}{{\mathbf I}}
\newcommand{\bfL}{{\mathbf L}}
\newcommand{\bfM}{{\mathbf M}}
\newcommand{\bfzero}{{\boldsymbol{0}}}

\newcommand{\Adele}{\textsf{\upshape A}}
\newcommand{\Ahat}{\hat{A}}
\newcommand{\Adot}{A(\Adele_K)_{\bullet}}
\newcommand{\Aut}{\operatorname{Aut}}
\newcommand{\Br}{\operatorname{Br}}  
\newcommand{\Closure}{\textsf{\upshape C}} 
\newcommand{\Disc}{\operatorname{Disc}}
\newcommand{\Div}{\operatorname{Div}}
\newcommand{\End}{\operatorname{End}}
\newcommand{\Fbar}{{\bar{F}}}
\newcommand{\Mbar}{{\overline{\Mcal}}}
\newcommand{\FOD}{\textup{FOM}}
\newcommand{\FOM}{\textup{FOD}}
\newcommand{\Gal}{\operatorname{Gal}}
\newcommand{\GL}{\operatorname{GL}}
\newcommand{\Index}{\operatorname{Index}}
\newcommand{\into}{\hookrightarrow}
\newcommand{\Image}{\operatorname{Image}}
\newcommand{\liftable}{{\textup{liftable}}}
\newcommand{\hhat}{{\hat h}}
\newcommand{\Ksep}{K^{\textup{sep}}}
\newcommand{\Ker}{{\operatorname{ker}}}
\newcommand{\Lsep}{L^{\textup{sep}}}
\newcommand{\Lift}{\operatorname{Lift}}
\newcommand{\LS}[2]{{\genfrac{(}{)}{}{}{#1}{#2}}} 
\newcommand{\plim}{\operatornamewithlimits{\text{$p$}-lim}}
\newcommand{\wlim}{\operatornamewithlimits{\text{$w$}-lim}}
\newcommand{\MOD}[1]{~(\textup{mod}~#1)}
\newcommand{\Norm}{{\operatorname{\mathsf{N}}}}
\newcommand{\notdivide}{\nmid}
\newcommand{\normalsubgroup}{\triangleleft}
\newcommand{\odd}{{\operatorname{odd}}}
\newcommand{\onto}{\twoheadrightarrow}
\newcommand{\Orbit}{\mathcal{O}}
\newcommand{\ord}{\operatorname{ord}}
\newcommand{\Per}{\operatorname{Per}}
\newcommand{\PrePer}{\operatorname{PrePer}}
\newcommand{\PGL}{\operatorname{PGL}}
\newcommand{\Pic}{\operatorname{Pic}}
\newcommand{\Prob}{\operatorname{Prob}}
\newcommand{\Qbar}{{\bar{\QQ}}}
\newcommand{\rank}{\operatorname{rank}}
\newcommand{\Resultant}{\operatorname{Res}}
\renewcommand{\setminus}{\smallsetminus}
\newcommand{\Span}{\operatorname{Span}}
\newcommand{\tors}{{\textup{tors}}}
\newcommand{\Trace}{\operatorname{Trace}}
\newcommand{\twistedtimes}{\mathbin{%
   \mbox{$\vrule height 6pt depth0pt width.5pt\hspace{-2.2pt}\times$}}}
\newcommand{\UHP}{{\mathfrak{h}}}    
\newcommand{\Vdot}{V(\Adele_K)_{\bullet}}
\newcommand{\Wreath}{\operatorname{Wreath}}
\newcommand{\<}{\langle}
\renewcommand{\>}{\rangle}

\newcommand{\Hgt}{\mathrm{Height}}
\newcommand{\hgt}{\mathrm{height}}
\newcommand{\wi}{\mathrm{wideg}}
\newcommand{\longhookrightarrow}{\lhook\joinrel\longrightarrow}
\newcommand{\longonto}{\relbar\joinrel\twoheadrightarrow}
\newcommand{\Ob}{\operatorname{Ob}}

\newcommand{\Stab}{\operatorname{Stab}}
\newcommand{\Spec}{\operatorname{Spec}}
\renewcommand{\div}{{\operatorname{div}}}

\newcommand{\reduce}[1]{\widetilde{#1}}
\newcommand{\helpme}[1]{{\sf $\heartsuit\heartsuit$ Meta-remark: [#1]}}
\newcommand{\fixme}[1]{{\sf $\spadesuit\spadesuit$ Meta-remark: [#1]}}
\newcommand{\series}[2]{#1[\![#2]\!]}
\newcommand{\laurent}[2]{#1(\!(#2)\!)}

\newcounter{CaseCount}
\Alph{CaseCount}
\def\Case#1{\par\vspace{1\jot}\noindent
\stepcounter{CaseCount}
\framebox{Case \Alph{CaseCount}.\enspace#1}
\par\vspace{1\jot}\noindent\ignorespaces}

\title[\TITLERUNNING]{\TITLE}
\date{\DATE}

\author{Liang-Chung Hsia and Hua-Chieh Li}
\noindent \address{Department of Mathematics,
         National Taiwan Normal University,
         Taipei, Taiwan, R. O. C.}
\email{hsia@math.ntnu.edu.tw}
\noindent\address{Department of Mathematics,
         National Taiwan Normal University,
         Taipei, Taiwan, R. O. C.}
\email{li@math.ntnu.edu.tw}

\thanks{The first author was partially supported by NSC Grant  102-2115-M-003-002-MY2 and
  he also acknowledges the support from NCTS.   The second author was
  partially supported by NSC Grant 102-2115-M-003-001.}


\begin{abstract}
Let  $G\subset x\series{\FF_q}{x}$ ($q$ is a power of the prime $p$) be a subset of formal power series over a finite field such that it forms a compact abelian $p$-adic Lie group of dimension $d\ge 1$. We establish a necessary and sufficient condition for the APF extension of local field corresponding to $\left(\laurent{\FF_q}{x}, G\right)$ under the field of norms functor to be an extension of $p$-adic fields.  We then apply this result to study family of invertible power series with coefficients in a  $p$-adic integers ring and  commute with a fixed noninvertible power series  under the composition of power series.
\end{abstract}

\maketitle

\section*{Introduction}

The purpose of this note is to study families $G \subset x\series{R}{x}$ of formal power series
 in the  cases where $R$ is  a finite field or a $p$-adic integer ring. In both cases, $G$ is a commutative group under  the operation of compositions of power series.
 In the former case,  $G$  is $p$-adic Lie group of finite dimension whose opposite $G^{0}$ acts on the Laurent series field over $R$ faithfully.
 In the latter case, $G$ consists of all  invertible power series that commute with a fixed  stable, noninvertible power series (see Section~\ref{section:Lubin's Conjecture} for   definitions).

More specifically, in the case where $R$ is a finite field, we're concerned with  $G$ being a closed subgroup of the ramification group $\Ncal(R) = x + x^2 \series{R}{x}$  such that it is isomorphic to $\ZZ_p^d$ for some integer $d\ge 1$  where $p$ is the characteristic of $R.$
It follows from the theory of field of norms that there exists an arithmetically profinite Galois extension $L/K$ of local fields such that $\laurent{R}{x}$ is the field of norms corresponding to the extension $L/K$ and $G$ is isomorphic to the Galois group of the extension.  In this case we say that $L/K$ is the arithmetically profinite extension corresponding to the pair $\left(\laurent{R}{x}, G\right)$ under the field of norms functor.
It is an interesting and in many applications, important question  about the characteristic of the fields $L/K.$ In the case where $G\simeq \ZZ_p$ ($d =1$), let $\s$ be a topological generator of $G$ then  Wintenberger~\cite{win04} shows that $K$  is a $p$-adic field if and only if the limit $\lim_{n\to \infty} i_n(\s)/p^n $ exists. Here $i_n(\s) = i(\s^{p^n})$ denotes the $n$-th ramification number of $\s$ (see \S~\ref{subsec:hasse-herbrand}).

It is natural to seek for a generalization of this result to the case where $G$ has higher $\ZZ_p$ rank. Our first main result gives one such generalization
 which we state as follows.

\begin{Theorem}
Let $\FF_q$ denote the finite field of $q$ elements ($q$ is a power of $p$).
Let $G\simeq \ZZ_p^d$  be a closed abelian subgroup of $\Aut_{\FF_q}\left(\laurent{\FF_q}{x}\right).$
Let $L/K$ be an abelian extension corresponding to $\left(\laurent{\FF_q}{x}, G\right)$
under the field of norms functor. Then  $K$ is a $p$-adic field
if and only if  there exists a constant $\k$ depending on $G$ such that for every non-identity $\s\in G$ we have
\begin{equation*}
\frac{i_{n+2}(\s) - i_{n+1}(\s)}{i_{n+1}(\s) - i_n(\s)} = p^d
\quad\text{for all $n\ge \k$}.
\end{equation*}

\end{Theorem}

Let $K$ be a finite extension of $\QQ_p$ and let $\Ocal_K$ be its ring of integers. A formal power series $f(x)\in x\series{\Ocal_K}{x}$  is called invertible if $f'(0)\in \Ocal_K^{\ast}$; otherwise it is called noninvertible. Let $g(x)\in x\series{\Ocal_K}{x}$ be a noninvertible power series with nonzero linear coefficient $g'(0).$  In the paper~\cite{Lub}, Lubin studies the following question:  suppose that there exists an invertible power series $u$ which commutes with $g$ under the composition of power series, i.e. $g\circ u = u\circ g,$ what we can say about the two power series $u$ and $g$?
 Assuming that  the $n$-th iterates $u^{\circ n}$ of $u$ are not the identity for all positive integer $n,$ Lubin suggests that the commutative pair of power series $g$ and $u$ are related to certain formal groups defined over $\Ocal_K.$ Since Lubin's paper was published, there have been some work related to his question, see Section~\ref{section:Lubin's Conjecture} for detailed account.

In this paper, we study the case where $g$ has large set of invertible power series that commute with it. More precisely, we
let $G$ be the set of all invertible power series in $x\series{\Ocal_K}{x}$ commuting with  $g.$ By Lubin's result~\cite[Corollary~1.1.1]{Lub}, the map $u(x)\mapsto u'(0)$ gives rise to an injective group homomorphism from $G$ (under the  operation of composition) to $\Ocal_K^{\ast}.$ We assume that this homomorphism is also surjective. Then, in this case  Lubin's question can be rephrase as: whether or not there exists a Lubin-Tate formal group $\Gcal(x,y)$ over $\Ocal_K$ such that $g$ is an endomorphism and $G$ is the group of automorphism of $\Gcal(x,y).$

Our second main result (Theorem~\ref{thm:lubin's conjecture})  is to give an affirmative answer to this question under the condition that $g'(0)$ is a uniformizer of $\Ocal_K$ or $K$ is an unramified extension of $\QQ_p.$ Before stating the result, we let $v_K : K^{\ast}\to \ZZ$ be the normalized discrete valuation on $K$ and denote by $\partial_0$  the homomorphism sending $u(x)\in G$ to $u'(0)\in \Ocal_K^{\ast}.$

\begin{Theorem}
Let $K$ be an finite extension over $\QQ_p$  with ramification index $e$ and residue degree $f$. Suppose that $g(x)\in x\series{\Ocal_K}{x}$ is a stable noninvertible series with Weierstrass degree equal to  $f \cdot v_K(g'(0)).$ Furthermore, assume that  $\partial_0\left(G\right) =  \Ocal_K^{\ast},$  then $g$ is an endomorphism of a Lubin-Tate formal group defined over $\Ocal_K$ if one of the following conditions holds:
\begin{enumerate}
\item  $v_K(g'(0))=1$ ($e$ can be any positive integer).
\item $e=1$ (i.e. $K$ is an unramified extension over $\QQ_p$) and all the roots of iterates of $g(x)$ are simple.
\end{enumerate}
\end{Theorem}

We remark that Sarkis~\cite{sarkis10} applies ideas in~\cite{LMS} to  prove that for $g(x)\in x\series{\ZZ_p}{x}$ under the condition that $v_p(g'(0)) = 1$, Weierstrass degree  equal to $p$ and $\partial_0(G) = \ZZ_p^{\ast}$, there exists a formal group over $\ZZ_p$ such  that $g$ is an endomorphism and $G$ is the group of automorphisms of the formal group. Our result extends his to more general situations. One of the new ingredients here is our generalization of Wintenberger's result mentioned above. We refer the reader to Section~\ref{section:Lubin's Conjecture} for details.

The plan of our paper is as follows. In Section~\ref{section:preliminary}, we give an overview of ramification subgroup of automorphisms of a local field of positive characteristic. We also review some facts about arithmetically profinite extensions as well as the construction of the field of norms. After these preliminaries, we prove our first main result in Section~\ref{section:APF}. We divide the statement  and proof of our result  into two parts. We first prove the necessary condition in Theorem~\ref{thm:char0lowernumbergrowth} and then prove the sufficient condition in Theorem~\ref{thm:char0converse}. Section~\ref{section:Lubin's Conjecture} is devoted to the study of Lubin's conjecture. In this long section, we give a possible detailed overview on Lubin's conjecture. After some preparations in \S~\ref{subsec:reduction} and \S~\ref{subsec:apply FON}, we begin the proof for our second main result in \S~\ref{subsec:pf Lubin conj} and finally we give some remarks in \S~\ref{subsec:final remarks}.


\section{Preliminaries}
\label{section:preliminary}
\subsection{Wild ramification}
\label{subsec:hasse-herbrand}

Let $k$ be a perfect field of characteristic $p>0$ and let $\Gcal_0(k)$ denote the set of power series in $\series{k}{x}$ whose leading term is of degree one. Then $\Gcal_0(k)$ is a group under the compositions of power series and  $\Gcal_0(k)^{op}$ acts on $\laurent{k}{x}$ by substitution $g\mapsto g\circ f$ for $g\in \laurent{k}{x}$ and $f\in \Gcal_0(k)$ . Moreover, this gives an isomorphism between $\Gcal_0(k)^{op}$ and $\Aut_k\left(\laurent{k}{x}\right)$. The subgroup $\Ncal(k) = x + x^2 \series{k}{x}$ of $\Gcal_0(k)$,  known as the Nottingham group or ramification subgroup of $\Aut_k\left(\laurent{k}{x}\right)$, is a pro-$p$ group.

Recall that the ramification number of $\s \in \Aut_k\left(\laurent{k}{x}\right)$ is defined by
$$i(\sigma)=\ord_x(\sigma(x)-x) - 1.$$
By convention, we set $i(\s) = \infty$ for $\s(x) = x$ (the identity automorphism). The automorphism $\s$ is called a wildly ramified automorphism if $i(\s) \ge 1.$ Equivalently, $\s$ is wildly ramified if and only if $\s \in \Ncal(k).$ In this case, we set $i_m(\s) = i(\s^{p^m})$ for $m\ge 0$.  Then  an elementary fact about the sequence $\{i_m(\s)\}$ is that it is a strictly increasing sequence of integers if $\s$ is not of finite order. It was proved by S.~Sen~\cite[Theorem~1]{sen69} that for wildly ramified automorphism $\s$ of infinite order, we have  $i_{n}(\s) \equiv i_{n-1}(\s) \pmod{p^n}$ for all $n>0$. As a consequence, if $\s$ is not of finite order, then $i_n(\s)$ grows at least exponentially in $n$, namely

\begin{equation}
\label{eq:in lower bound}
i_n(\s) \ge \sum_{j=0}^n p^j.
\end{equation}

Notice that  the ramification numbers  $i(\cdot)$ gives rise to a natural
filtration on $\Gcal_0(k).$ 
In particular, for any closed subgroup  $G$ of $\Gcal_0(k)$  
its {\em lower  numbering} is given by
\[
G[s] := \{\s\in G \mid i(\s)\ge s\}\quad \text{for all $s\in\RR$}.
\]
It follows from the definition that $G[s] = G$ for $s\le 0.$  In the case where $G[s]$ is of finite index in $G$ for all $s\in \RR$, the following integral is well-defined  
\[
\p_G(s) = \int_0^s \,\frac{dt}{\left(G : G[t]\right)}\quad \text{for
  all $s\in \RR$}.
\]
The function $\p_G$ is a piecewise linear and increasing function. Its inverse is denoted by $\ps_G$ which in this paper we call the Hasse-Herbrand function for $G$. Then
the {\em upper numbering} of $G$ is defined via $\ps_G$.
Namely, $G(y) := G[\ps_G(y)]$ for all $y\in \RR.$
Thus,  $G[s] = G(\p_G(s))$ for all $s\in \RR$. Moreover, we have
\[
\ps_G(y) = \int_0^y \,\left(G: G(t)\right) dt .
\]
The following is a  basic property that follows from the definition of lowering numbering.

\begin{proposition}
\label{proposition:lower-upper-numbering}
Let $G$ be a closed subgroup of $\Aut_k\left(\laurent{k}{x}\right)$
and let $H$ be a closed subgroup of $G$.
Then $H[s] = H \cap G[s]$ for any $s\in \RR.$
\end{proposition}

A real number $s$ is called a {\em break} (or jump) of the lower numbering for
$G$ if $G[s] \ne G[s+\ve]$ for any $\ve > 0.$ Likewise, if $G(t)\ne
G(t+\ve)$ then $t$ is a break of upper numbering. It follows from the
definition that the set of breaks of the lower numbering for $G$ is a
subset of integers. We'll denote the set of breaks of lower numbering
by $\{l_m\mid m=0,1, \ldots\}$ and   the set of breaks of upper
numbering by $\{u_m\mid m=0,1, \ldots\}.$ We have  $u_m = \p_G(l_m)$
for all integers $m \ge 0.$

\begin{remark}
\label{lem:subgroupbreaks}
Let $H$ be a closed subgroup of  $G$ and let $\ell$ be a break of the lower numbering of $H$. It follows from  Proposition~\ref{proposition:lower-upper-numbering} that $\ell$ is also a break of the lowering numbering of $G.$
In particular, if $H\simeq \ZZ_p$ is generated by $\s\in G$, then the set $\{i_n(\s)\mid n=0, 1, \ldots\}$ is just the set of
breaks of the lower numbering for $H$. Consequently, the sequence $\{i_n(\s)\}$ is a subset of the breaks of the lower numbering of $G.$

\end{remark}

\subsection{Arithmetically profinite extensions and field of norms}
\label{subsec:field of norms}
In this subsection, we summarize some basic facts from the theory of field of norms which was first investigated by
J.-M. Fontaine and J.-P. Wintenberger~\cite{FW79I, FW79II}. Let $K$ be a local field with perfect residue field
$\widetilde{K}$ of characteristic $p$ and let $L$ be an infinite {\em arithmetically  profiinite} (APF) extension of $K$ (see~\cite{win83}). Let  $X_K(L)$ denote the field of norms of $L/K$.  The  multiplicative group of $X_K(L)$  is given by
\[
X_K(L)^{\ast} = \varprojlim_{E\in \Ecal_{L/K}} E^{\ast}
\]
where $\Ecal_{L/K}$ is the set of all finite subextensions of $L/K$ and the inverse limit is taken with respect to the norms $N_{F/E} : F^{\ast}\to E^{\ast}$ for finite subextensions $F \supseteq E \supseteq K$ of $L/K$. Hence, $X_K(L) = X_K(L)^{\ast}\cup \{0\}$ and the set of nonzero elements of $X_K(L)$ consists of all norm-compatible sequences
$$\left\{\left(a_E\right)_{E\in \Ecal_{L/K}} \mid a_E \in E \;\text{and}\; a_E = N_{F/E}(a_F) \;\text{if}\; E \subseteq F\right\}.$$
In fact, $X_K(L)$ is a local field of characteristic $p$ whose residue field $\widetilde{X_K(L)}$ is isomorphic to $\widetilde{L}.$

It follows from   the construction that the field of norms establishes a faithful functor between the category of infinite APF extensions of $K$ with morphisms consisting of finite separable $K$-embeddings and the category of local fields of characteristic $p$ with morphisms consisting of finite separable embeddings~\cite[\S3]{win83}.
In the case where the APF extension $L$ is Galois over $K$, Wintenberger~\cite[Corollary~3.3.4]{win83} shows that the Galois group $G_{L/K}$ acts on $X_K(L)$ faithfully. Therefore, $G_{L/K}$ is isomorphic to  a closed subgroup $X_K(G_{L/K})$ of the group $\Aut(X_K(L))$ of continuous automorphisms of $X_K(L).$ Moreover, this identification preserves upper ramification subgroups. Namely, we have $X_K(G_{L/K}(t)) = X_K(G_{L/K})(t)$ for all real number $t$ where $G_{L/K}(t)$ denotes the upper ramification subgroup of $G_{L/K}$ with $t$ an upper numbering for $G_{L/K}$.

\subsection{$p$-adic Lie extensions}
\label{subsec:Lie extension}

Let $K$ be a local field and let $L/K$ be a Galois extension.  The extension field $L$ is called a $p$-adic Lie extension of $K$ if the Galois group $G_{L/K}$ of $L/K$ is a  $p$-adic Lie group of finite dimension. In the case where $K$ is a $p$-adic local field, S.~Sen~\cite{sen72} proves that $L$ is an APF extension of $K.$
Wintenberger~\cite{win80} shows that there is an equivalence  between the category of abelian $p$-adic Lie extensions and the category of local fields of characteristic $p$ equipped with a continuous action by compact abelian $p$-adic  Lie groups.

More precisely,  let the two categories $\Acal$ and $\Rcal$ be defined as follows.
Objects of category $\Acal$ are infinite Galois  extensions $L/K$ whose Galois group $G_{L/K}$ is an abelian $p$-adic  Lie group and the residue field $\widetilde{L}$ is a finite extension of $\widetilde{K}.$ An $\Acal$-morphism from $L/K$ to $L'/K'$ is defined to be a continuous embedding $\s : L \to L'$ such that $L'$ is a finite separable extension of $\s(L)$ and $K'$ is a finite extension of $\s(K).$

An object of $\Rcal$ consists of couples $(X, G)$ where $X$ is a complete local field of characteristic $p$ with perfect residue field and $G$ is a compact abelian $p$-adic  Lie group of  finite dimensional which is a closed subgroup  of the group $\Aut(X)$ of continuous automorphisms of $X$. An $\Rcal$-morphism $(j_1, j_2)$ from $(X,G)$ to $(X',G')$ consists of continuous embedding $j_1 : X \to X'$  such that $X'$ is a finite separable extension of $j_1(X)$ and continuous group homomorphism $j_2 : G' \to G$ such that $g'\circ j_1 = j_1 \circ j_2(g')$ for all $g'\in G'.$ Let $\Fcal : \Acal \to \Rcal$ be the functor such that $\Fcal(L/K) = (X_K(L), X_K(G_{L/K}))$ for $L/K\in \Ob(\Acal)$. Then,

\begin{theorem}[Wintenberger~\cite{win80}]
\label{thm:win80}
$\Fcal$ is an equivalence of categories.
\end{theorem}

\begin{remark}
In the following, we'll simply say that $L/K$ is an  extension corresponding to $(X,G)$ under the field of norms functor if  $(X,G)= \Fcal(L/K)$ for $(X, G)\in \Ob(\Rcal)$.
\end{remark}

Notice that the Galois extension $L/K$ is an APF extension of $K$ if $G_{L/K}(t)$ are open in $G_{L/K}$ for all real $t.$  In this case, the Hasse-Herbrand function
\[
\psi_{L/K}(s) = \int_0^s \left(G_{L/K} : G_{L/K}(t)\right) \, d t
\]
is well-defined for all $s\ge -1$. Let $\p_{L/K}$ be the inverse of $\psi_{L/K}.$ Then, the lowering numbering of $G_{L/K}$ is defined by $G_{L/K}[s] = G_{L/K}(\p_{L/K}(s))$ for all real $s\ge -1.$ The  following result is due to  Laubie~\cite{laubie88}.

\begin{proposition}
\label{prop:matching lower-number}
Let $L/K\in \Acal$ and let $\Fcal(L/K) = (X, G)$. Then, for all $x\ge -1$ we have $G_{L/K}[x] \simeq G[x].$
\end{proposition}

\section{Arithmetic profinite extensions over field of characteristic 0}
\label{section:APF}

By Theorem~\ref{thm:win80}, for any given pair $(X, G)\in \Ob(\Rcal)$ there exists an abelian $p$-adic Lie extensions $L/K$ corresponding to $(X,G)$ under the field of norms functor. In this section we're concerned with whether or not the field $K$ is of characteristic zero. We'll consider the case where $k$ is the finite field $\FF_q$ of $q$ elements  ($q$ a power of $p$) and $X = \laurent{\FF_q}{x}$ is the local field in question.

Let $G$ be a compact  abelian $p$-adic  Lie subgroup of $\Aut_{\FF_q}\left(\laurent{\FF_q}{x}\right)$.  In the special case where $G \simeq \ZZ_p$ generated by $\sigma \in G$ we write the limit  $\lim_{n\to \infty} \left(i_n(\sigma)/p^n\right) = \left(p/(p-1)\right) e.$ Wintenberger~\cite[Th\'eor\`eme~1]{win04} shows that either $e$ is a positive integer or $e =\infty.$ Moreover, $e$ is a positive integer if and only if  $L/K$ corresponding to $(X, G)$ is an extension of $p$-adic field. In fact, $e$ is the absolute ramification index of $K.$ The main goal of this section is to generalize Wintenberger's result to the general case where the $\ZZ_p$-rank of $G$ is an arbitrary positive integer.

\subsection{A necessary condition}
\label{subsec:necessary condition}

For  Lie extensions $L/K$ of $p$-adic local fields, a key ingredient for studying the extension $L/K$  is the following result  due to S.~Sen which gives the effect of taking $p$-th power on the upper ramification subgroups of the Galois group $G_{L/K}$.

\begin{proposition}[S.~Sen~\mbox{\cite[Proposition~4.5]{sen72}}]
\label{prop:p-adic Lie ext}
Let $K$ be a $p$-adic local field with absolute ramification index $e.$ Let $L/K$ be a $p$-adic Lie extension of $K$. Then for real number $t$ large enough, $G_{L/K}(t)^p = G_{L/K}(t + e)$.
\end{proposition}

We first present a necessary condition for the extension $L/K$ corresponding to $\left(\laurent{\FF_q}{x}, G\right)$ to be an $p$-adic fields.

\begin{theorem}
\label{thm:char0lowernumbergrowth}
Let $G$ be a closed abelian subgroup of $\Aut_{\FF_q}\left(\laurent{\FF_q}{x}\right)$
which is isomorphic to $\ZZ_p^d$ with $d \ge 1$. Let
$L/K$ be an abelian extension corresponding to $\left(\laurent{\FF_q}{x}, G\right)$
under the field of norms functor. If $K$ is a $p$-adic field
then there exists a constant $\k$ depending on $G$ such that for every non-identity $\s\in G$ we have
\begin{equation}
\label{eqn:in ratio}
\frac{i_{n+2}(\s) - i_{n+1}(\s)}{i_{n+1}(\s) - i_n(\s)} = p^d
\quad\text{for all $n\ge \k$}.
\end{equation}
\end{theorem}

\begin{proof}
Let $L/K$ be the abelian extension of $p$-adic field corresponding to the pair
$\left(\laurent{\FF_q}{x}, G\right)$ under the field of norms functor.
Let $e = v_K(p)$ be the absolute ramification index of $K.$
Denote  the sequences of breaks of lower and upper ramification numbers of $G$ by $\{l_m\}$ and $\{u_m\}$ respectively. Let  $i_n = i_n(\s)$ be the $n$-th ramification number of non-identity element $\s\in G.$ Our first goal is to show that for all $n$
large enough, $i_n$ satisfy~\eqref{eqn:in ratio}.

By Wintenberger~\cite[Corollary~3.3.4]{win83} and Laubie~\cite{laubie88} (see also Proposition~\ref{prop:matching lower-number}), the upper and lower numberings  are preserved between $G_{L/K}$ and $G$  under the field of norms functor. Therefore, the conclusion of Proposition~\ref{prop:p-adic Lie ext} applies to $G$ as well. That is,  we have $G(y)^p = G(y+e)$
for $y$ sufficiently large. Let $Y$ be a fixed number such that $G(y)^p = G(y+e)$
for all $y\ge Y.$ It  follows from~\cite{sen72} that $G(Y)$ is of
finite index in $G.$ Let $N = \left( G : G(Y)\right)$ which is a power of $p$ since  $G \simeq \ZZ_p^d$.
Write $N = p^{\k}$ for some nonnegative integer $\k.$

Let $\s\in G$ be a non-identity element. We put $i_m = i_m(\s)$ and set $\om_m = \p_G(i_m)$ for all integer $m\ge 0.$
Then, we have  $G(\om_m) = G(\p_G(i_m)) = G[i_m]$. Let  $n \ge \k$ be given. First we observe that $\om_n \ge  Y$. Indeed, since $n\ge \k$  we must have $\s^{p^n}\in G(Y)=G[\ps_G(Y)]$. By the definition of the $n$-th ramification number of $\s$ it follows that $i_n \ge \ps_G(Y).$ Applying $\p_G$, we get $\om_n = \p_G(i_n) \ge \p_G(\ps_G(Y)) = Y$ as desired.

Now Proposition~\ref{prop:p-adic Lie ext} says that $G(\om_n)^p = G(\om_n + e).$
Use the facts that $\s^{p^{n+1}} \in G[i_n]^p = G(\om_n)^p=G(\om_n+e)$ and $G(\om_n+e) = G[\ps_G(\om_n+e)]$,
we see that $i_{n+1}\ge \ps_G(\om_n+e).$ By applying the increasing function $\p_G$,   we conclude that $\om_{n+1} \ge \om_n +e.$ We claim that $\om_{n+1} = \om_n + e.$   It remains to show that $\om_{n+1} \le \om_n + e.$

By Proposition~\ref{prop:p-adic Lie ext} again, we have
 $G(\om_{n+1}) = G(\om_{n+1} - e + e)= G(\om_{n+1} - e)^p$ since $\om_{n+1} - e \ge \om_n \ge Y$. Therefore, $\s^{p^{n+1}}\in G[i_{n+1}] =  G(\om_{n+1}) =  G(\om_{n+1} - e)^p.$ On the other hand,  $G$ does not have nontrivial $p$-torsion elements, we thus conclude  that $\s^{p^n}$ is the unique element whose $p$-th power
is $\s^{p^{n+1}}.$ Hence  $\s^{p^n} \in G(\om_{n+1} - e)=G[\psi_G(\om_{n+1}-e)].$ Therefore, $i_n \ge \psi_G(\om_{n+1}-e)$. It follows that $\om_n = \p_G(i_n) \ge (\p_G\circ\ps_G)(\om_{n+1}-e) = \om_{n+1} - e$ which is the desired inequality.
 Hence we must have $ \om_{n+1} = \om_n + e $ for $n\ge \k.$  

By the definition of Hasse-Herbrand function for $G$, we have
\begin{equation*}
i_{n+1} - i_n = \psi_{G}(\om_{n+1}) - \psi_{G}(\om_n) = \int_{\om_n}^{\om_{n+1}} \, \left(G : G(t)\right) dt.
\end{equation*}
Therefore,
\begin{align*}
i_{n+2} - i_{n+1} & = \int_{\om_{n+1}}^{\om_{n+2}} \, \left(G : G(t)\right) dt \\
                   & = \int_{\om_n+e}^{\om_{n+1}+e} \, \left(G : G(t)\right) dt \\
                   & = \int_{\om_n}^{\om_{n+1}} \, \left(G : G(t+e)\right) dt  \\
                   & = \int_{\om_n}^{\om_{n+1}} \, \left(G : G(t)^p\right) dt \\
                   & = \int_{\om_n}^{\om_{n+1}} \, \left(G :
                     G(t)\right) \left(G(t) : G(t)^p\right) dt \\
                   & = \int_{\om_n}^{\om_{n+1}} \, p^d \left(G : G(t)\right) dt \\
                   & = p^d \left(i_{n+1} - i_n\right) \quad \text{for all $n \ge \k.$}
\end{align*}
That is, the sequence $\{i_n\}$ satisfies the  relation
$$
\frac{i_{n+2} - i_{n+1}}{i_{n+1} - i_n} = p^d
$$
for all $n \ge \k$ as desired.
\end{proof}

\begin{remark}
\label{rmk:in closed form}
It's not hard to see that~\eqref{eqn:in ratio} is equivalent to the following closed form  of $i_n(\s)$  for $\s\in G.$
\begin{equation}
\label{eq:growthofin}
i_n(\sigma)=i_{\k}(\sigma)+\frac{p^{d(n-\k)}-1}{p^d-1}
(i_{\k+1}(\sigma)-i_{\k}(\sigma))\quad\text{for all $n\ge \k$}.
\end{equation}
\end{remark}

\subsection{A sufficient condition}
\label{subsec:height}

For a pro-$p$ group $G$, when we say the sequence $\{\t_n\}$ of
elements of $G$ converges
to $\t\in G$, we mean the convergence is with respect to the pro-$p$
topology on $G.$
We'll denote by $\plim \t_n = \t$. To ease the notation, we use
$G_n = G^{p^n}$ to denote the subgroup of $p^n$-th power of $G$ for
positive integer $n.$ In the following, we'll denote the identity of $G$ by
$1$ if there's no danger of confusion.

\begin{proposition}
\label{prop:continuityofi}
Let $G$ be a ramification subgroup which  is isomorphic to $\ZZ_p^d$ ($d\ge 1)$.
\begin{parts}
\item[\upshape(i)]
Let $\{\t_n\}$ be a sequence of elements in $G$ such that $\plim \t_n = \t$. Then,
\[\left\{\begin{array}{ll}
 \lim_{n\to\infty}i(\tau_n)=\infty , & \hbox{if $\tau =  1$;} \\
i(\tau_n)=i(\tau)\; \mbox{for $n$ sufficiently large,} &
\hbox{otherwise.}\end{array}\right.\]

\item[\upshape(ii)]
For any positive integer $N$, there exists an upper bound $B_N$ such that
$i(\s) \le B_N$ for all $\s\in G\setminus G_N.$

\end{parts}
\end{proposition}

\begin{proof}
We first notice that  if $\plim_{n\to\infty} \t_n = \t$ then for any given integer $m$, there exists
a positive integer $N$ such that for all $n\ge N$ we have $\t_n \t^{-1} \in G_m.$
This is because the family of subgroups $\{G_m\}$ forms a fundamental system
of open neighborhood of the identity element of $G.$

To prove (i), let $\t_n$ be a given sequence which has a limit $\t$ in $G$. Let
$m$ be a positive integer. As remarked above, there exits
an integer $N = N(m)$ such that $\t_n \t^{-1} \in G_m$ for all $n\ge N.$
Fix a set of generators $\{\s_1,\ldots, \s_d\}$ for $G$. We see that
$i(\t_n\t^{-1}) \ge \min \{i_m(\s_1),\ldots, i_m(\s_d)\}$. By~\eqref{eq:in lower bound}, the number
$\min \{i_m(\s_1),\ldots, i_m(\s_d)\}$ is unbounded as $m\to \infty.$ Consequently
$i(\tau_n\tau^{-1})\rightarrow\infty$ as $n\to \infty.$

Write  $i(\t_n) = i(\t_n \t^{-1} \t) .$ We first consider the case where  $\t = 1.$
Then for any integer $m$ there is an $N$ such that $\t_n = \t_n \t^{-1}\in G_m$ for all $n > N.$
Hence $i(\t_n) = i(\t_n\t^{-1}) \to \infty.$ Next, we assume that $\t\ne 1.$
Because $\lim_{n\to \infty} i(\t_n\t^{-1}) = \infty$, there exists an integer $T$
such that $i(\t_n\t^{-1}) > i(\t)$ provided that $n> T$. It follows that
$i(\t_n) = \min\{i(\t), i(\t_n\t^{-1})\} = i(\t)$ for all $n > T$  as desired.

For (ii), we fix a positive integer $N.$ Notice that $G/G_N$ is a finite group.
Let $\{\l_1,\ldots, \l_r\}$ be a fixed coset representatives of $ G_N$ in $G.$
Suppose that there exists a sequence $\{\tau_n\}$ such that  $\tau_n\not\in G_N$ for all $n$ and
$i(\tau_n)\rightarrow\infty$ as $n\to \infty.$ Then there exists an infinite subsequence of
$\{\tau_n\}$ in the coset $\lambda_iG_N$ with $\lambda_i\ne 1$ for some $i\in\{1,\dots,r\}.$.
Since $\lambda_iG_N$ is compact (in the pro-$p$ topology),
this subsequence has a limit point $\tau$ in $\lambda_iG_N$. Hence, without lose
of generality we may assume that $\plim_{n\to\infty} \tau_n=\tau$. Since $\tau\ne 1$,
it follows from (i) that $i(\tau_n)=i(\tau)$ for $n$ sufficiently large.  This contradicts to
the assumption that $i(\tau_n)\rightarrow\infty$. Hence, the function $i(\cdot)$ is bounded above
in $G\setminus G_N$ by a constant depending on $G_N$ only.

\end{proof}

It is natural to ask whether or not the relation~\eqref{eqn:in ratio} in Theorem~\ref{thm:char0lowernumbergrowth}(or ~\eqref{eq:growthofin}  in  Remark~\ref{rmk:in closed form})
is also sufficient to guaranteed that the extension $L/K$ corresponding to $\left(\laurent{\FF_q}{x}, G\right)$ is an
extension $p$-adic field. The following result which is the converse to Theorem~\ref{thm:char0lowernumbergrowth}
gives an affirmative answer.

\begin{theorem}
\label{thm:char0converse}
Let $G\subseteq\Aut_{\FF_q}(\laurent{\FF_q}{x})$ be an abelian group which is isomorphic
to $\ZZ_p^d$. Suppose that there exists a $\k\ge 1$ such that for all $n\ge \k$,
 the $n$-th ramification number $i_n(\s)$  satisfies~\eqref{eqn:in ratio}
for every  element non-identity element $\s\in G$,  that is
\begin{equation*}
\frac{i_{n+2}(\s) - i_{n+1}(\s)}{i_{n+1}(\s) - i_n(\s)} = p^d.
\end{equation*}
 Then, the field extension corresponding to $\left(\laurent{\FF_q}{x}, G\right)$ under the field of norms functor is an extension of $p$-adic field.
\end{theorem}

\begin{remark}
In~\cite{laubie98, laubie00}, the authors studied the growth of the $n$-th lower numberings of a set of generators for $G$ under special restrictions on the initial ramification numbers of the generators. With the given conditions, the corresponding extension $L/K$ under the field of norms functor is shown to be extension of $p$-adic field. Also, the growth of the $n$-th lower numberings (for generators) as in~\eqref{eq:growthofin} are  obtained. However, the methods used in their  paper do not seem applicable to the general situation.
\end{remark}

\begin{remark}
It is reasonable to ask in the general case  whether or not~\eqref{eqn:in ratio} satisfied by a set of generators for $G$ is sufficient for the corresponding extensions to be $p$-adic field. We do not have an answer to this question in general.

\end{remark}

The proof of Theorem~\ref{thm:char0converse} is based on the following proposition which is a corollary
to the theorem of Wintenberger~\cite[th\'eor\`em~3.1]{win80} giving criteria for the corresponding extensions to be of characteristic 0.

\begin{proposition}
 \label{prop:criterion of zero char}
Let $G$ be the group corresponding to the Galois group of the
arithmetic profinite extension $L/K$ under the field of norm
functor. Then the following conditions are equivalent.
\begin{parts}
  \item[(i)]  the characteristic of $K$ is 0; \\
  \item[(ii)]  $0 < \liminf_{x\to \infty} \frac{x}{(G : G[x])} < \infty$ and
  $0 < \limsup_{x\to \infty} \frac{x}{(G : G[x])} < \infty$.
\end{parts}

\end{proposition}

\begin{proof}[Proof of Theorem~\ref{thm:char0converse}]
Let $L/K$ be the extension corresponding to $\left(\laurent{\FF_q}{x}, G\right)$ under the field of norm functor.
Notice that any subgroup which is of finite index in $G$ corresponds to a
subextension $L/M$ of $L/K$ with $[M : K]$ finite. Therefore, by restricting to subgroups
$G[m]$ for $m$ large enough, we may assume that the constant $\k=0$ in~\eqref{eqn:in ratio}. Then, as indicated in
Remark~\ref{rmk:in closed form}, this is equivalent to assuming that  for
all non-identity elements $\s\in G$ we have
\begin{equation}
\label{eq:growthofinconverse}
i_n(\sigma)=i_{0}(\sigma)+\frac{p^{dn}-1}{p^d-1}
(i_{1}(\sigma)-i_{0}(\sigma))\quad\text{for all $n\ge 0$}.
\end{equation}
 By Proposition~\ref{prop:criterion of zero char} (ii), we need
  to show that the ratio $x/(G : G[x])$ has finite positive
  supremum and infimum as $x\to \infty.$

Let $x$ be a given positive real number. Notice that $G$ is a free
$\ZZ_p$ module of rank $d$ and $G[x]$ as a closed
subgroup of $G$ with respect to the pro-$p$ topology is  a $\ZZ_p$-submodule.
By the theory of  module over PID, for every given real number $x$ there exists
a $\ZZ_p$-basis $\{\s_1,\ldots, \s_d\}$ for $G$ so that
a $\ZZ_p$-basis for $G[x]$  is of the form $\{\s_1^{p^{n_1}},
\ldots, \s_d^{p^{n_d}}\}$ and hence  $\left(G : G[x]\right) = p^{\sum_j n_j}$.
On the other hand, for every $j = 1, \ldots, d$ we have
\begin{equation}
\label{eq:in interval}
 i_{n_j - 1}(\s_j) < x \le i_{n_j}(\s_j)\quad \text{.}
\end{equation}
To ease the notation, for each $j$ we write $i_n(\s_j) = \a_j p^{dn} + \b_j$
where
\begin{align*}
\a_j  & = \frac{i_1(\s_j) - i_0(\s_j)}{p^d - 1}\quad \text{and} \\
 \b_j & = \frac{p^d\cdot i_0(\s_j) - i_1(\s_j)}{p^d-1}.
\end{align*}
Substitute these identities back to~\eqref{eq:in interval} and divide~\eqref{eq:in interval} by $p^{d n_j}$ we get
\[
\a_j p^{-d} + \b_j p^{-dn_j}  <  \frac{x}{p^{dn_j}} \le \a_j + \b_j p^{-d n_j}\quad j=1,\ldots, d
\]
and
\[
   \left[\prod_{j=1}^d \left(\a_j p^{-d} + \b_j p^{-d n_j}\right)\right]^{1/d} <
   \frac{x}{p^{\sum_j n_j}} \le
   \left[\prod_{j=1}^d\left(\a_j + \b_j p^{-d n_j}\right)\right]^{1/d}.
 \]
Since $\s_1,\ldots, \s_d$ generates $G$, we must have $\s_j\not\in
G^p$ for every $j=1,\ldots, d.$ Although  $\a_j$ and $\b_j$ depend on the choice of $\s_j$ for each $j$ and the choice of the basis $\{\s_1,\ldots, \s_d\}$ depends on the given real number $x$,
Proposition~\ref{prop:continuityofi} shows that $i_0(\s_j), i_1(\s_j)$
 and hence $\a_j$ and $\b_j$ are bounded above by a constant which is independent of the given real number $x$ for all
 $j=1,\ldots, d.$ Therefore, we conclude that  there exist positive constants $U$
 and $L$ which are independent of $x$ such that
 \[
0 <   L <    \frac{x}{p^{\sum_j n_j}} <  U .
 \]
 This implies that
 \[
 L < \frac{x}{\left(G : G[x]\right)} < U
   \]
   for all $x$ sufficiently large. Hence,
\[
0< L \le \liminf_{x\to\infty}\frac{x}{\left(G : G[x]\right)} \le
  \limsup_{x\to\infty}\frac{x}{\left(G : G[x]\right)} \le U <\infty.
\]
Now, it follows from Proposition~\ref{prop:criterion of zero
  char} that the field $M$ and hence $K$ are $p$-adic fields.
\end{proof}

\section{Application--Commuting power series}
\label{section:Lubin's Conjecture}

In~\cite[\S6]{Lub}, Lubin made a conjecture concerning families of commuting power series defined over a $p$-adic integer ring. The conjecture suggests that  for an invertible series  to commute with a noninvertible series, there must be a formal group somehow in the background. Although this conjecture is still far from completely understood, there has been some interesting work on  this conjecture, see~\cite{LMS,li, li1, li2, Lub, sarkis10}.
In~\cite{LMS}, Laubie,  Movahhedi and Salinier apply the theory of field of norms to the study of Lubin's conjecture and their idea was later applied by Sarkis in~\cite{sarkis10} to prove a special case of  Lubin's conjecture.
Inspired by the work in~\cite{LMS, sarkis10}, we will apply our results in Section~\ref{section:APF} to the characterization of commuting families of power series proposed by Lubin's conjecture.  First, let's set the following notation which will remain fixed throughout this section.

\begin{notation}
\setlength{\itemsep}{3pt}
\item[$K$]
a finite extension of $\QQ_p$ with degree $d = [K : \QQ_p].$
\item[$f,\, e$]
the residue degree and ramification index of $K$ respectively.
\item[$\Ocal_K$]
the  ring of integers of $K$.
\item[$\Mcal_K$]
the  maximal ideal of $\Ocal_K.$
\item[$U_K$]
the units group of $\Ocal_K.$
\item[$\pi$]
a fixed uniformizer of $\Ocal_K$ such that $\Mcal_K = \pi \Ocal_K$.
\item[$v_K$]
the normalized valuation on $K$ such that $v_K(\pi) = 1.$
\item[$\FF_q$]
 the residue field of $\Ocal_K$ which is a finite field of
$q$ elements with $q = p^f$.
\item[ $\Dcal_0(\Ocal_K)$]
 the set of formal power series over $\Ocal_K$ without constant term.
\item[$\wi(g(x))$]
the Weierstrass degree of the power series $g(x).$
\end{notation}
\medskip

Notice that $\Dcal_0(\Ocal_K)$ is a monoid under the operation of substitution. Following the notation and terminologies in~\cite{Lub}, the composition of two power series is denoted by  $(f\circ g)(x) = f(g(x))$ and  the $n$-th iterates of $g(x)$ is denoted by $g^{\circ n}(x)$. A power series $g\in \Dcal_0(\Ocal_K)$ is called {\em invertible} if it has an inverse under the operation of substitution; otherwise it is called {\em noninvertible}. It's an elementary fact that a power series $g\in \Dcal_0(\Ocal_K)$ is  invertible  if and only if $g'(0)\in U_K$.
We denote the subset of invertible power series  by $\Gcal_0(\Ocal_K).$ Then, $\Gcal_0(\Ocal_K)$ forms a group under the operation of substitutions.
Moreover, $\Gcal_0(\Ocal_K)$ acts on $\Dcal_0(\Ocal_K)$ by conjugation. Let $g\in \Dcal_0(\Ocal_K)$, the stabilizer of $g$  under the action of  $\Gcal_0(\Ocal_K)$ is denoted by
\[
\Stab_{\Ocal_K}(g) := \{u\in \Gcal_0(\Ocal_K)\mid u\circ g \circ u^{-1} = g\}.
\]
In other words, $\Stab_{\Ocal_K}(g)$ is the subset of $\Dcal_0(\Ocal_K)$ consisting of all invertible power series that commute with $g$ under the composition of power series.

\begin{definition}
Let $g\in \Dcal_0(\Ocal_K)$. We say that $g$ is a {\em torsion} (series) if $g$ is of finite order under substitution; we say $g$ is {\em stable} if $g'(0)$ is not 0 nor a root of 1.
\end{definition}

\noindent By~\cite[Corollary~1.1.1]{Lub} we have the following.

\begin{proposition}
\label{prop:multiplier}
Let $g\in \Dcal_0(\Ocal_K)$ and let $\partial_0 : \Stab_{\Ocal_K}(g) \to U_K$  be the map defined by
$\partial_0(u) = u'(0)$. If $g(x)$ is a  stable power series, then $\partial_0$ gives an injective group homomorphism from $\Stab_{\Ocal_K}(g)$ into $U_K.$
\end{proposition}

Let $g\in \Dcal_0(\Ocal_K)$ be a stable noninvertible  series and assume that $g$ is an endomorphism of a Lubin-Tate formal group $F(x,y)$ over $\Ocal_K.$ Then $\Stab_{\Ocal_K}(g)$ is the units group of the ring $\End_{\Ocal_K}(F(x,y))$ of endomorphisms of $F(x,y).$ Moreover  the homomorphism $\partial_0$ as in Proposition~\ref{prop:multiplier} is surjective.
It is natural to ask if the converse also holds. More precisely, suppose that $\partial_0(\Stab_{\Ocal_K}(g))= U_K$, must $g$  come from an endomorphism of a suitable Lubin-Tate formal group defined over $\Ocal_K$?

In the case where $K = \QQ_p$ and $g\in\Dcal_0(\ZZ_p)$ is a noninvertible power series,  assuming that $v_p(g'(0)) =1, \wi(g(x)) = p$ and  $\partial_0(\Stab_{\ZZ_p}(g)) = \ZZ_p^{\ast},\,$  Sarkis~\cite{sarkis10}  proves that there exists a formal group $F(x,y)$ over $\ZZ_p$ such that $g$ is an endomorphism of $F$ and $\Stab_{\ZZ_p}(g)$ is the group of automorphisms of $F.$
His result provides a support evidence for Lubin's conjecture in the special case where $K = \ZZ_p$ and the stabilizer group $\Stab_{\ZZ_p}(g)$ of  $g(x)$ is the largest possible.

The main goal in this section is to generalize Sarkis result  to a more general situation. Namely, assuming that $\partial_0(\Stab_{\Ocal_K}(g)) = U_K$ for the given noninvertible power series $g$ we show that indeed $g$ is an endomorphism for a Lubin-Tate formal group over $\Ocal_K$ provided that some conditions are satisfied by $K$ and $g$ (see Theorem~\ref{thm:lubin's conjecture} below).
Before stating our result, we recall the definition   as introduced in~\cite{li-proceedings} for the {\em height}  of  a stable noninvertible  series $g,$
\[\hgt(g)=e\cdot \frac{\log_p(\wi(g(x)))}{v_K(g'(0))}.\]

The motivation for defining the height of a noninvertible power series comes from that of a formal group $F(x,y)\in\series{\Ocal_K}{x,y}$ which is defined to be $\log_p(\wi([p]_F(x))\!)$ where $[p]_F$ is the endomorphism of the formal group $F(x,y)$ with first degree coefficient equal to $p$. It turns out that the height of $F(x,y)$ is equal to $\hgt(f)$  for any nonzero, noninvertible endomorphism $f(x)$ of $F(x,y)$. Thus, if $g(x)$ is a noninvertible endomorphism of a formal group, then $\hgt(g)$ is equal to the height of the formal group.
 Notice that in the case treated by Sarkis,  the power series $g(x)$ has height one.

 The following is the main result in this section.
\begin{theorem}
\label{thm:lubin's conjecture}
Let $K$ be an finite extension over $\QQ_p$  with ramification index $e$ and residue degree $f$. Suppose that $g(x)\in \Dcal_0(\Ocal_K)$ is a stable noninvertible series of $\hgt(g) = ef.$ Furthermore, assume that  $\partial_0\left(\Stab_{\Ocal_K}(g)\right) =  U_K,$  then $g$ is an endomorphism of a Lubin-Tate formal group defined over $\Ocal_K$ if one of the following conditions holds:
\begin{enumerate}
\item  $v(g'(0))=1$ ($e$ can be any positive integer).
\item $e=1$ (i.e. $K$ is an unramified extension over $\QQ_p$) and all the roots of iterates of $g(x)$ are simple.
\end{enumerate}

\end{theorem}

\begin{remark}
\label{remark:simple roots}
It is well known that for a noninvertible endomorphism of a formal group, every root of its iterates is simple. On the other hand, there exist ``condensation'' of  endomorphisms of a formal group which are noninvertible power series whose roots of iterates are not always simple (\cite[Examples 1 and 2]{iso}). There is also an example of noninvertible power series with infinite Weierstrass degree  whose stabilizer group contains non-torsion elements (\cite[Last Example in Section 5]{pty}).
In this note, we concentrate only on noninvertible stable power series with finite Weierstrass degree whose roots of iterates are simple. Note that in case~(1) of Theorem~\ref{thm:lubin's conjecture}, all  roots of iterates of $g(x)$ must be simple.
\end{remark}

The strategy for proving Theorem~\ref{thm:lubin's conjecture} follows the lead in~\cite{LMS}. One of the new ingredients here is Theorem~\ref{thm:char0converse} which gives a criterion for determining  the characteristic of the fields $N/M$ corresponding to $\left(\laurent{\FF_q}{x}, G\right)$ under the field of norms functor, where the group $G$ is the reduction of $\Stab_{\Ocal_K}(g)$ modulo the maximal ideal $\Mcal_K$ of $\Ocal_K.$
Notice that unlike the cases treated in~\cite{LMS, sarkis10} where $G\simeq \ZZ_p$, our group $G$ here has $\ZZ_p$-rank equal to $d = [K : \QQ_p]$ which is greater than 1 in general. This causes difficulties in extending their method  to the general situation.

The reason is that one needs to know whether or not $N/M$ is an extension of $p$-adic field. If the $\ZZ_p$-rank of $G$ is equal to 1, then Wintenberger's result~\cite[Th\'eor\`eme~1]{win04} is the key to show that the characteristic of the fields $N/M$ are 0. Our generalization (Theorem~\ref{thm:char0converse}) of Wintenberger's result in the case where $G$ has higher $\ZZ_p$-rank makes it possible to apply the theory of field of norms to the study of Lubin's conjecture.

Yet in order to apply Theorem~\ref{thm:char0converse}, it is important to know that elements in  $\Stab_{\Ocal_K}(g)$ satisfy~\eqref{eqn:in ratio}. The second named author's previous work~\cite{li,li1} (see Lemma~\ref{lem:height} and Lemma~\ref{lem:findi_n} below) provide a way of computing the $n$-th ramification number for elements in $G$ and the absolute ramification index for the field $F.$ Once we have these ingredients in hand, we're able to show that under the conditions in Theorem~\ref{thm:char0converse},  $g(x)$ is indeed an endomorphism of a Lubin-Tate formal group over $\Ocal_K.$

We postpone the proof of Theorem~\ref{thm:lubin's conjecture} to \S~\ref{subsec:pf Lubin conj}. Before the proof, we study the reduction $G$ of the group $\Stab_{\Ocal_K}(g)$ and the extension fields corresponding to $\left(\laurent{\FF_q}{x}, G\right)$ under the field of norms functor.

\subsection{Reduction}
\label{subsec:reduction}

For $\a\in \Ocal_K$ we use $\reduce{\a}\in \Ocal_K/\Mcal_K$ to denote the reduction of $\a$ modulo $\Mcal_K.$ The reduction map $\series{\Ocal_K}{x} \to \series{\FF_q}{x}$ is as usual defined  by reducing the coefficients of a power series $f(x)$ modulo the maximal ideal $\Mcal_K.$ We'll simply use  $\reduce{f}(x)$ to denote the reduction of the power series $f(x).$

\begin{lemma}
\label{lem:reduction injective}
Let $\Gcal_0(\FF_q)$ denote the group of invertible power series over $\FF_q$. Then, the reduction map induces an injective group homomorphism $\Stab_{\Ocal_K}(g) \into \Gcal_0(\FF_q)$.
\end{lemma}

\begin{proof}
The lemma follows from~\cite[Corollary~4.3.1]{Lub}.
\end{proof}

We recall below some well-known facts about the units group $U_K$ which will be used in the sequel.
\smallskip
\begin{itemize}
\item $U_K$ is a finitely generated $\ZZ_p$-module of rank $d = [K : \QQ_p].$
\item For integer $i\ge 0$, we set $U^{(i)} =\{\alpha\in U_K\mid \alpha\equiv 1\pmod{\Mcal_K^{i}}\}.$ Then, $U^{(r)}$ is a free $\ZZ_p$ module of rank $d$ provided that  $r>e/(p-1)$.
\item For integer $\ell > e/(p-1)$, we have $\left(U^{(\ell)}\right)^p = U^{(\ell + e)}.$
\end{itemize}
From now on, we  fix a  stable noninvertible  series  $g\in \Dcal_0(\Ocal_K)$ such that
\[ \hgt(g)= d,\quad
      \partial_0(\Stab_{\Ocal_K}(g))= U_K
\]
  and all its roots of iterates are simple.

  To ease the notion, we'll write $\GG = \Stab_{\Ocal_K}(g)$ in the following.  Fix an integer $r >e/(p-1)$. It follows from above that $U^{(r)}$ is a free $\ZZ_p$-module of rank $d$.  Let
  \smallskip
  \begin{itemize}
\item $G =  \{\reduce{u} \mid u\in \GG\},$
\item $\GG^{(n)} = \partial_0^{-1}\left(U^{(n)}\right),$
\item $G^{(n)} = \{\reduce{u} \mid u\in \GG^{(n)}\}.$
\end{itemize}
By Proposition~\ref{prop:multiplier} and Lemma~\ref{lem:reduction injective} we have $G \simeq \GG\simeq \partial_0\left(\GG\right) = U_K.$
We extend the definition by setting $G^{(x)} = G^{(n)}$ for any real number $x\ge 0$ such  that  $n-1< x \le n.$  Then $G$ is viewed as a closed subgroup of $\Aut_{\FF_q}\left(\laurent{\FF_q}{x}\right)$ by the operation of substitution on the right to elements of $\laurent{\FF_q}{x}$. Moreover, we have that $G^{(r)}\simeq \ZZ_p^d$ and
$G^{(r)}$ is a closed subgroup of the ramification group $\Ncal(\FF_q)$.
It follows from the definition that  for $\s(x) =\reduce{u}(x) \in G\cap \Ncal(\FF_q)$, it's ramification numbers is related to the Weierstrass degrees of iterates of $u(x)$ by the formula
$$i_n(\s) = \wi(u^{\circ p^n} (x) - x) -1 .$$ In particular, this applies to elements of $G^{(r)}.$

As a consequence of the assumption on the simplicity of   roots of iterates of $g(x),$  we  can  apply the following two results which are the key ingredients  to apply Theorem \ref{thm:char0converse} to our situation.

\begin{lemma}[H.-C. Li~\protect{\cite[Theorem 3.9]{li}}]
\label{lem:height}
There exists $R$ and $\lambda>0$ (depending only on $g(x)$) such that for every $\s(x) \in G\cap \Ncal(\FF_q),$
\[\frac{i_{n+1}(\sigma)-i_{n}(\sigma)}{i_{n}(\sigma) - i_{n-1}(\sigma)}=p^\lambda\]
for all $n\ge R.$
\end{lemma}

The following Lemma  enables us to calculate the ramification numbers of a given element $\s(x)\in G\cap\Ncal(\FF_q).$

\begin{lemma}[H.-C. Li~\protect{\cite[Corollary 4.1.1]{li1}}]
\label{lem:findi_n}
There exists a constant $R$ (depending only on $g(x)$) with the following property.
Let $\s = \reduce{u}(x)\in G\cap \Ncal(\FF_q)$. Suppose that there exist integers   $m$ satisfying  $m v(g'(0))=v((u'(0))^{p^n}-1)$ for some $n\ge R$, then $\wi(u^{\circ p^n}(x) - x) = \wi(g^{\circ m}(x))$ and hence
$i_n(\sigma) + 1 =\wi(g^{\circ m})$.
\end{lemma}

\subsection{Applying the field of norms functor}
\label{subsec:apply FON}

By Theorem~\ref{thm:win80}, there exists a totally ramified Galois APF extension of local fields corresponding to
$\left(\laurent{\FF_q}{x}, G\right)$ under the field of norms functor. Let's denote this extension by $N/M$.
In the following, we determine the characteristic of $M$ first.

\begin{proposition}
\label{prop:lubin char 0}
The  characteristic of $M$  is 0.
\end{proposition}

\begin{proof}
Since any closed subgroup of $G = \reduce{\GG}$ of finite index corresponding to the extension $N/M'$ for some finite extension $M'$ of $M$ and the characteristic of the fields does not change, we may  consider the extension corresponding to $\left(\laurent{\FF_q}{x}, G^{(r)}\right)$ for some fixed $r>e/(p-1)$. In this case, $G^{(r)} \simeq \ZZ_p^d$.

Let $s = v_K(g'(0))\ge 1$ and let $\s = \reduce{u}\in G^{(r)}$ be any non-identity element. The assumption $\hgt(g)=d$ implies $\wi(g)=p^{sf}=q^s$. Put $\ell = v_K(u'(0) -1).$ Then $u'(0) = 1 + \ve \pi^{\ell}$ for some $\ve \in U_K.$  Since  $u'(0)\in U^{(r)}$ and $e = v_K(p)$, we have
\[(u^{\circ p^n})'(0)=(u'(0))^{p^n}\equiv 1+p^n(\ve \pi^{\ell})\pmod{\Mcal_K^{\ell+ne+1}}.\]

By assumption, $g$ satisfies one of the two conditions in Theorem~\ref{thm:lubin's conjecture}. Therefore either $e=1$ or $s=1.$ In either case, there exist positive integers $n$ and $m$ such that
\[v_K((u^{\circ p^n})'(0)-1)=v_K(u'(0)-1)+v_K(p^n)= \ell + n e = ms.\]
Moreover, we can choose $n$ sufficiently large so that Lemma~\ref{lem:findi_n} applies. Now Lemma~\ref{lem:findi_n} implies that
\begin{equation*}
\wi(u^{\circ p^n}(x)-x)=\wi(g^{\circ m})=q^{ms}
\end{equation*}
and  we also have
\begin{equation}
\label{eq:computing widegree}
\wi(u^{\circ p^{n+j s}}(x)-x)=\wi(g^{\circ (m+j e)})=q^{(m+j e)s} \quad \text{for all integers $j\ge 0$.}
\end{equation}

By Lemma~\ref{lem:height} there exists a positive integer $R$ and $\l > 0$ such that for all $\s\in G^{(r)}$,
the sequence $\{i_t(\s)\}_{t\ge R}$ satisfies  recursive relation
\[\frac{i_{t+1}(\sigma)-i_{t}(\sigma)}{i_{t}(\sigma) - i_{t-1}(\sigma)}=p^\lambda.\]
We claim that $\l = d.$

To prove the claim, we choose the integer $n$ such  that $n\ge R$ and~\eqref{eq:computing widegree} holds.  For $j\ge 0$ we have
\begin{align*}
  i_{n+j s}(\sigma) & =  q^{(m+ j e)s}-1.
\end{align*}
Now we compute $i_{n+j s}(\s) - i_{n+(j-1)s}(\s) =  q^{(m+ (j-1 ) e)s}(q^{es} -1)$ for all $j\ge 1.$ Therefore,
\[ \frac{i_{n+2s}(\s) - i_{n+s}(\s)}{i_{n+s}(\s) - i_n(\s)} = q^{es} = p^{ds}. \]
On the other hand, the recursive relation in Lemma~\ref{lem:height} satisfied by $i_n(\s)$ gives the following
\begin{align*}
  i_{n+j s}(\sigma) & = i_{n}(\sigma)+\frac{p^{js \l }-1}{p^\l-1}(i_{n+1}(\sigma)-i_{n}(\sigma)).
\end{align*}
It follows that
\[ \frac{i_{n+2s}(\s) - i_{n+s}(\s)}{i_{n+s}(\s) - i_n(\s)} = p^{\l s}. \]
From these, we conclude that $\l = d.$

As $G^{(r)}\simeq \ZZ_p^d$ and Lemma~\ref{lem:height} shows that all elements of $G^{(r)}$ satisfy~\eqref{eqn:in ratio},
it follows from Theorem~\ref{thm:char0converse} that  $M'$  and hence $M$ is of characteristic zero.

\end{proof}

Notice that, the residue  fields $\reduce{M} =\reduce{N}=  \FF_q = \reduce{K}.$ Hence, the
residue degree of $M$ is also equal to the residue degree $f$ of $K.$  Our next task is to compute the (absolute)
ramification index $e_M = v_M(p)$ of $M$ where $v_M$ denotes the normalized valuation on $M$.

\begin{proposition}
\label{prop:ramification index computation}
The ramification index of $M$ is the same as that of $K.$ That is, $e_M = e.$
\end{proposition}

\begin{proof}
As in the proof of Proposition~\ref{prop:lubin char 0}, we fix a positive integer $r > e /(p-1)$ and let $N/M'$ be the extension corresponding to $\left(\laurent{\FF_q}{x}, G^{(r)}\right)$ under the field of norms functor. Then $M'$ is a totally ramified extension over $M$ of degree $[M' : M] = [G : G^{(r)}].$ The strategy is to  compute the ramification index of $M'$ first. Then it's straightforward that  $e_{M} = e_{M'}/[M':M]$ since $M'/M$ is a totally ramified extension.

Now that $N/M'$ is an extension of  local fields of characteristic zero, Theorem~\ref{thm:char0lowernumbergrowth} shows that  nontrivial elements of $G^{(r)}$ satisfy~\eqref{eq:growthofin}. We fix a sufficiently large positive integer $R$ such that it works in  Lemma~\ref{lem:findi_n} and for every nontrivial $\s\in G^{(r)}$, its $n$-th ramification number $i_n(\s)$ satisfies~\eqref{eq:growthofin}  for all $n\ge R.$

Let $\s = \reduce{u}\in G^{(r)}$ be a nontrivial element, then
\[
i_n(\s) = i_R(\s) + \frac{p^{d(n-R)}-1}{p^d - 1}\left(i_{R+1}(\s) - i_R(\s)\right) \quad \text{for all  $n\ge R$.}
\]
Put $\ell = v_K(u'(0)-1)$,  $s= v_K(g'(0))$ and choose an $n_0\ge R$ such that $\ell + n_0 e = m s$ for some positive integer $m$. By~\eqref{eq:computing widegree}, we have
\[i_{n_0+js}(\s) = \wi(g^{\circ (m+je)s}) - 1= q^{(m+je)s} - 1.
 \]
 Substituting $i_{n_0}(\s)$ and $i_{n_0+s}(\s)$ into the above equations and a little manipulation on algebraic identities, we obtain that
 \begin{equation}
 \label{eqn:in for sigma}
 i_n(\s) + 1 = q^{\ell + ne} \text{ for all $n\ge R$.}
 \end{equation}

 Notice that~\eqref{eqn:in for sigma} implies that $\log_q(i_n(\s) + 1) = v_K((u^{\circ p^n})'(0) - 1)$ for all $\s = \reduce{u}\in G^{(r)}$ and all $n\ge R$. On the other hand, for any given positive integer $n$ there exists  sufficiently large integer $r$ such that  $U^{(r)}\subseteq U^{p^n}$. That it, elements of $U^{(r)}$ are of the form $u^{\circ p^n}$ for some $u\in G.$ Therefore, by increasing $r$ if necessary, we may assume that the ramification number of every nontrivial $\s = \reduce{u} \in G^{(r)}$ satisfies $\log_q(i(\s) + 1) = v_K(u'(0) - 1).$ Then by the definition of  lower numberings,  for $t\ge 0$
\[
G^{(r)}[t] = \{\reduce{u}\in G^{(r)}\mid v_K(u'(0)-1) \ge \log_q(t+1)\}.
\]

Since  $G^{(\ell)} = \{\reduce{u}\mid v_K(u'(0)-1) \ge \ell\}$ for $\ell \ge r$, we see that $G^{(r)}[t] = G^{(\log_q(t+1))}.$  Now the isomorphism $G^{(\ell)} \simeq U^{(\ell)}$ shows that the filtration
\[
G^{(\ell)}\supsetneqq G^{(\ell+1)}\supsetneqq \cdots \supsetneqq  G^{(\ell+j)} \supsetneqq  \cdots
\]
is the same as that arising from ramification numbers of elements of $G^{(\ell)}$. Hence,
the breaks of the lower numberings of $G^{(r)}$ are exactly the set of $t$ such that $t = q^{\ell}-1$ for all  integers $\ell\ge r.$

Moreover, we have $[G^{(r)} : G^{(\ell)}] = [U^{(r)} : U^{(\ell)}] = q^{\ell-r}. $
Notice that  we also have $\left(G^{(\ell)}\right)^p = G^{(\ell + e)}.$ Since $G^{(\ell)} = G^{(r)}[q^\ell - 1]$, it follows that
\[
\left(G^{(r)}[q^{\ell} - 1]\right)^p = G^{(r)}[q^{\ell + e} - 1].
\]

To ease the notation in the computation below, we'll simply put $H = G^{(r)}.$
Applying the function $\p_{H}$ and Proposition~\ref{prop:p-adic Lie ext}, we have
\begin{align*}
e_{M'} & = \p_{H}\left(q^{\ell + e} - 1\right) - \p_{H} \left(q^{\ell}-1\right) \\
     & = \int_{q^{\ell}-1}^{q^{\ell+e}-1}\,\frac{d t}{\left[H : H[t]\,\right]} \\
     & = \sum_{j=0}^{e-1}\, \int_{q^{\ell+j}-1}^{q^{\ell+j+1}-1} \,\frac{d t}{\left[H : H[t]\,\right]}  \\
     & = \sum_{j=0}^{e-1} \, \frac{q^{\ell+j+1} - q^{\ell+j}}{\left[H : H[q^{\ell+j+1}-1]\,\right]}  \\
     & = \sum_{j=0}^{e-1} \, \frac{q^{\ell+j}(q - 1)}{q^{\ell+j+1-r}} \\
     & = q^{r-1}(q-1) e.
\end{align*}

As remarked in the beginning of the proof, $e_M = e_{M'}/[M' : M] = e_{M'}/[G : G^{(r)}].$ Now,  $[G : G^{(r)}] = [U_K : U^{(r)}] = q^{r-1}(q-1).$ Combining with the  formula for $e_{M'}$ obtained above, we conclude that $e_M = e$ as desired.
\end{proof}

\begin{remark}
\label{rmk:the same degree}
Since the two fields $K$ and $M$ have the same residue degrees and ramification indices, it follows that
both fields have the same degrees over $\QQ_p.$ That is, we have  $[M: \QQ_p] = e f = [K : \QQ_p].$
\end{remark}

Now we know that the extension $N/M$ corresponding to $\left(\laurent{\FF_q}{x}, G\right)$  under the functor
of the field of norms  is a
totally ramified  extension of $p$-adic field  with Galois group $G_{N/M} \simeq G$. Our next result gives a more precise description of the extension $N/M.$

\begin{proposition}
\label{prop:max totally ramf ext}
The extension $N/M$ is a maximal totally ramified abelian extension of $M.$
\end{proposition}

\begin{proof}
As $N$ is a  totally ramified abelian extension of $M$, by local class field theory $N/M$ is a subextension of a maximal
totally ramified  abelian extension $N'/M$  of $M$. From the theory of Lubin-Tate formal groups, we know that $N'$ is the field generated by all the $\pi_M^n$-torsion points (for all $n\in \NN$) over $M$ of a Lubin-Tate formal group $F(x,y)\in\series{\Ocal_M}{x,y}$ corresponding to a uniformizer $\pi_M$ of $M.$

Recall that the formal group $F(x,y)$ gives rise to  an isomorphism  $U_M \to G_{N'/M}$ where $U_M$ denotes the units group of $\Ocal_M.$ On the other hand,  the restriction map $\s\mapsto \s|_N$ for  $\s\in G_{N'/M}$ induces  a  surjective continuous group homomorphism $G_{N'/M}\to G_{N/M}.$  Composing with the isomorphisms $U_M\simeq G_{N'/M}$ and $G_{N/M}\simeq G\simeq U_K$, we obtain a surjective  homomorphism  $\th: U_M \twoheadrightarrow U_K$ of $\ZZ_p$-modules.  Our goal is to show that the kernel of $\th$, denoted by $W$, is trivial.

By Remark~\ref{rmk:the same degree},  the two $\ZZ_p$-modules $U_K$ and $U_M$  are of rank $d$.
Therefore $W$ must be a finite subgroup of $U_M.$  Since both $K$ and $M$ have the same residue degrees, the prime-to-$p$ part of the torsion groups of $U_K$ and $U_M$ have the same cardinality. If $W$ is not trivial then it
must be a finite $p$-group. Let the order of $W$ be $p^{\ell}$ for some $\ell\ge 0.$

For $a\in U_M$, we denote by $\gfk_{a} \in G_{N'/M}$ corresponding to $a$.  Notice that the isomorphism $G_{N/M}\simeq G$ is the one given by the field of norms functor $X_M(\cdot).$ Let $\g_a = X_M(\gfk_a|_N)$ be the image in $G$ of the restriction of $\gfk_{a}$ on $N.$ The  lower numbering $i(\g_a)$ is computed in~\cite[Proposition~3.10]{LMS}. We will show that $\ell = 0$ by comparing the lower numberings $i_n(\g_a)$ obtained in~\cite{LMS} and the one computed in Proposition~\ref{prop:ramification index computation}. For our purpose, it suffices to consider a special case as follows.

We set $h_n = \max_{w\in W}(v_M(a^{p^n} - w))$ for positive integer $n.$ Then~\cite[Proposition~3.10]{LMS} implies that there is a positive integer $e'$ depending on $W$ such that for $n$ large enough, we have
\begin{align*}
i_n(\g_a) & = \frac{q^{h_n} -p^{\ell}}{p^{\ell}} + C_W, \quad \text{where} \\
C_W & = \begin{cases}
\frac{p-1}{p}\sum_{\nu = 0}^{\ell-1} \frac{q^{e' p^\nu}}{p^\nu} & \text{if $\ell \ge 1$, }\\
    0 & \text{otherwise.}
\end{cases}
\end{align*}

By raising $a$ to a suitable power, we may assume that  $\g_a\in G^{(r)}$ and  $t_a = v_M(a -1) = \max_{w\in W}(v_M(a - w)) .$ Then, $h_n = v_M(a^{p^n} - 1) = t_a + n v_M(p) = t_a + n e.$ Recall that  $\g_a(x) = \reduce{u}(x)$ for some $u\in \GG$ with $u'(0) = b\in U_K.$   As $\g_a \in G^{(r)}$, by the definition of $G^{(r)}$ we also have that $b\in U_K^{(r)}.$  Let $t_b = v_K(b -1).$ Then~\eqref{eqn:in for sigma} implies that $i_n(\g_a) = q^{t_b + ne} - 1$ for $n$ sufficiently large. If $\ell \ne 0$,  combining both  results we see that for  $n$ sufficiently large
\[
q^{t_b + ne} - 1 = \frac{q^{t_a + ne} - p^{\ell}}{p^{\ell}} + \frac{p-1}{p}\sum_{\nu = 0}^{\ell-1} \frac{q^{e' p^\nu}}{p^\nu} .
\]
Equivalently,
\[
q^{ne}\left(q^{t_b} - \frac{q^{t_a}}{p^\ell}\right) = \frac{p-1}{p}\sum_{\nu = 0}^{\ell-1} \frac{q^{e' p^\nu}}{p^\nu}.
\]
This can not hold for all large $n$. Therefore, $\ell = 0$ and consequently $N' = N$ as desired.
\end{proof}

\begin{remark}
If $e =1$ then both $K$ and $M$  are unramified extension of degree $f$ over $\QQ_p$. It follows that $M = K$ in this case. Then, Proposition~\ref{prop:max totally ramf ext} follows from local class field theory directly.
\end{remark}

\subsection{The proof}
\label{subsec:pf Lubin conj}

We begin to prove Theorem~\ref{thm:lubin's conjecture}. Recall that we're given a stable series $g\in \Dcal_0(\Ocal_K)$ such that  $\GG = \Stab_{\Ocal_K}(g)$ satisfying $\partial_0(\GG) = U_K$. For the abelian group $G = \reduce{\GG}$, the APF extension  $N/M$
corresponding to $\left(\laurent{\FF_q}{x}, G\right)$ under the field of norms functor is a maximal totally ramified abelian extension of $p$-adic  local fields. Furthermore, the residue degree and ramification index of  $M$ are the same as that of $K.$

As in the proof of Proposition~\ref{prop:max totally ramf ext}, $N$ is the field generated by  torsion points over $M$ of a Lubin-Tate formal group law $F(X,Y)\in \series{\Ocal_M}{X,Y}$ associated to a uniformizer $\omega$ of $M$. It is well-known that the formal group law $F(X,Y)$ induces an isomorphism $\r_F : \Ocal_M\to \End(F(X,Y))$ such that for $\a\in \Ocal_M$ the image of $\a$ under $\r_F$ is the unique endomorphism $[\a]_F$ of $F(X,Y)$ satisfying $\partial_0([\a]_F) =  \a.$  Furthermore, via the formal group law $F(X,Y)$ the isomorphism $G_{N/M}\simeq U_M$ can be explicitly given as follows.

For $\t\in G_{N/M}$, there exists a unique  $\a\in U_M$ such that for any positive integer $n$ and any  $\omega^n$-torsion $\l$ of $F(X,Y)$ we have $\t(\l) = [\a]_F(\l).$ On the other hand, the result of~\cite[Lemma~3.2]{LMS} says that under the field of norms functor, for $\t\in G_{N/M}$ corresponding to $[\a]_F$ we have that $X_M(\t) = \reduce{[\a]_F} \in \Aut(X_M(N)).$ Here we also denote by $\reduce{[\a]_F}(x)$ the reduction of the formal power series $[\a]_F(x)\in \series{\Ocal_M}{x}$ modulo $\Mcal_M.$

By the construction of the field of norms (see \S~\ref{subsec:field of norms}), a uniformizer $x$ of $X_M(N)\simeq \laurent{\FF_q}{x}$ corresponds to a sequence of norm-compatible elements $(\pi_{M'})_{M'}$ where $M'$ runs through finite subextensions of $N/M$ with $\pi_{M'}$ a uniformizer in $M'.$  As in the proof of~\cite[Th\'{e}or\`{e}m~4.5, (iv)$\Rightarrow$ (v)]{LMS}, we may choose the Lubin-Tate formal group law $F(X,Y)$ to be the one corresponding to $\pi_M$.  Moreover, the same proof also shows  that $F(X,Y)$ can be chosen such that
$G =  \{\reduce{[\a]_F}(x) \mid \a \in U_M\}.$
In other words, if we let $\Fcal(X,Y)\in \series{\FF_q}{X,Y}$ be the reduction of the Lubin-Tate formal group $F(X,Y)$ modulo $\Mcal_M$, then
the group $G$ is a closed subgroup of $\Aut(\Fcal(X,Y))$.(Although~\cite[Th\'{e}or\`{e}m~4.5]{LMS} treats the case where $M=\QQ_p$ and $G_{N/M}\simeq \ZZ_p$ only, the arguments can be generalized to our case without difficulty).

Let $s = v_K(g'(0)) \ge 1.$ Then  $\wi(g)= p^{sf} = q^s$ and $\reduce{g}(x)=\gamma(x^{q^s})$ for some invertible series $\gamma(x)\in \series{\FF_q}{x}$ (\cite[Corollary 6.2.1]{Lub}). We claim that there exists an invertible power series $u(x) \in \GG$ such that $\reduce{u}(x) = \g(x).$

Notice that the reduction (modulo $\Mcal_M$) map  induces an injective ring homomorphism $\r_{\Fcal}: \Ocal_M\hookrightarrow \End(\Fcal)$ such that $G$ is the units group of $\r_{\Fcal}(\Ocal_M).$
For any  $\s\in G$ we have $\g\circ \s = \s\circ \g$ as well since $\s\circ \reduce{g} = \reduce{g}\circ \s$ and $\s$ commutes with the $q$-th power Frobenius map.  In particular, $\g$ commutes with any non-torsion element of $G$.  By~\cite[Theorem 6]{LubSar}, we conclude that
$\g \in \End(\Fcal).$ So, $\g$ is in the centralizer of $G$ in $\End(\Fcal).$  From this, it's not hard to deduce that $\g$ is in the centralizer of $\r_{\Fcal}(\Ocal_M)$ in $\End(\Fcal).$ It follows that $M(\g)$ is an extension field of $M$ contained in $D = \End(\Fcal)\otimes \QQ_p.$
On the other hand,  the endomorphism ring of $\Fcal(x,y)$ is a maximal order in the division algebra $D$ of rank $d^2$ over $\QQ_p$ (\cite[Theorem 20.2.13]{Haz}). Since $M$ is of degree $d$ over $\QQ_p,$ it is a maximal field contained in $D.$ Thus,  we must have $M(\g) = M.$ This implies that $\g$ is actually contained in $\r_{\Fcal}(\Ocal_M)$ and in fact, it is contained in $G$. Therefore, there exists a formal power series $u(x)\in \GG$  such that $\reduce{u}(x) = \g(x)$ as  claimed.

 The remaining proof is split into two cases according to: (i) $s = 1$ and (ii) $K$ is unramified over $\QQ_p$ ($e =1 $) .
\medskip
\\
{\bf Case (I) $s = 1$ :}
Put $h=g\circ u^{\circ (-1)}$. Then $h(x)\equiv x^q\pmod{\Mcal_K}$ and $v_K(h'(0))=v_K(g'(0))=1$. By Lubin-Tate theory, the formal power series $h(x)$ is an endomorphism of a Lubin-Tate formal group $G(x,y)$ over $\Ocal_K$, and as $g(x)$  commutes with $h(x)$, it is  an endomorphism of $G(x,y)$ as well. This complete the proof of the first case of Theorem~\ref{thm:lubin's conjecture}.
\medskip
\\
{\bf Case (II) The unramified case ($e = 1$):}
Notice that we have $M = K$  in this case since both $M$ and  $K$ are unramified extensions of the same degree over $\QQ_p$.
The embedding $\r_{\Fcal} : \Ocal_K \into \End(\Fcal(X,Y))$ makes $\Fcal(X,Y)$ into a formal $\Ocal_K$-module, denoted by $(\Fcal(X,Y), \r_{\Fcal})$, of $\Ocal_K$-height 1. (see~\cite[\S~21.8.2]{Haz} for a definition).
Since $\reduce{g}(x) = \g(x^{q^s})$ and $\g$ as well as the $q^s$-power map are endomorphisms of $\Fcal(x,y)$,
we conclude that $\reduce{g}$ is an endomorphism of $\Fcal(X,Y)$ which commutes with elements of $\r_{\Fcal}(\Ocal_K).$
It follows from the same arguments as above that in fact  $\reduce{g}\in \r_{\Fcal}(\Ocal_K).$
Hence there exists a unique $\b\in \Ocal_K$ such that $\reduce{g} = [\b]_{\Fcal}.$
Our goal is to show that there exists an invertible power series $\p(x)\in \series{\Ocal_K}{x}$ such that $\reduce{\p}(x) = x$ and $g(x)$ is an endomorphism of  the Lubin-Tate formal group law $F^{\p}(X,Y) = \p^{-1}(F(\p(X), \p(Y))).$

To achieve that goal we apply the lifting techniques  studied in~\cite{li2} to our situation.  We start with  the basic setting in~\cite{li2}. Let  $\zeta$ and $ \mu$ be fixed noninvertible and invertible endomorphism of $\Fcal(x,y)$ respectively such that $\z\circ \m = \m\circ\z .$ For every positive integer $r$, we write $M_r=\Mcal_K^r/\Mcal_K^{r+1}$ which is a one dimensional vector space over $\FF_q.$ Let's recall the definitions for the   $\FF_q$-subspaces $Z_r(\zeta,\mu), Z_r^o(\zeta,\mu)$ and $B_r(\zeta,\mu)$  of $\series{M_r}{x}\oplus \series{M_r}{x}$ in the following.
\begin{align*}
Z_r(\zeta,\mu) & =\{(d,w)\in\series{M_r}{x}\oplus \series{M_r}{x} \,:\, d(\mu(x))= d(x)\cdot
\mu'(\zeta(x))+ w(\zeta(x))\}. \\
Z_r^o(\zeta,\mu) & = \{(d,w)\in Z_r(\zeta,\mu) \,:\,d'(0)=w'(0)=0\; \text{ in $M_r$} \}. \\
B_r(\zeta,\mu) & = \{(d, w) \in Z_r(\zeta,\mu)  \,:\, d(x)= -\theta(\zeta(x)),
     \text{ for some $\theta(x)\in x^2 \series{M_r}{x}$} \}.
\end{align*}

We remark that if  $(d, w) \in Z_r(\z,\m)$  such that $d(x) = - \th(\z(x))$ for some $\th(x)\in x^2\series{M_r}{x},$  then
\begin{align*}
w(\z(x)) & = \m'(\z(x))\cdot \th(\z(x)) - \th(\z(\m(x))) \\
   & = \m'(\z(x)) \cdot \th(\z(x)) - \th(\m(\z(x))).
   \end{align*}
Consequently,
\begin{equation}
\label{eqn:fun eq for B_r}
w(x) = \m'(x)\cdot\th(x) -\th(\m(x))
\end{equation}
and  $B_r(\z, \m)$ is a subspace of $Z_r^o(\z,\m).$

A pair of formal power series $(f, u)\in \Dcal_0(\Ocal_K)\oplus \Gcal_0(\Ocal_K)$ is said to be a {\em lifting} of $(\zeta,\mu)$ provided that  $f\circ u=u\circ f$ and $\reduce{f}(x)=\zeta(x)$ and $\reduce{u}(x)= \mu(x)$. Notice that the definition for lifting given here is slightly different from the one  in~\cite{li2}.  Suppose that $(f_1,u_1)$ and $(f_2,u_2)$ are two liftings  of $(\zeta,\mu)$ such that
 \[f_1\equiv
f_2\pmod{\Mcal_K^r}\mbox{ and }u_1\equiv u_2\pmod{\Mcal_K^r}\]
then  $(f_1-f_2,u_1-u_2) \in Z_r(\z,\m).$

A result of~\cite[Section 3.2]{Sarkis05} says that there exists a  $\p_r(x)\in \Gcal_0(\Ocal_K)$ such that  $\phi_r(x)\equiv x \pmod{\Mcal_K^{r}}$ and
\[\phi_r\circ f_1\equiv f_2\circ \phi_r\pmod{\Mcal_K^{r+1}}\quad\mbox{and}\quad\phi_r\circ u_1\equiv u_2\circ \phi_r\pmod{\Mcal_K^{r+1}}\]
if and only if  $(f_1-f_2,u_1-u_2)\in B_r(\zeta,\mu)$. In fact, we have $f_1(x) - f_2(x) \equiv -\th_r(\z(x))\pmod{\Mcal_K^{r+1}}$ where $\th_r(x)$ is any power series in $x^2 \series{M_r}{x}$ such that $\th_r(x) \equiv \p_r(x) - x\pmod{\Mcal_K^{r+1}}$ (see~\cite[p.~141]{Sarkis05}).

We set  $\z(x) = [\b]_{\Fcal}(x).$ Clearly  $\reduce{[\b]_{F}}(x)  = \z(x) = \reduce{g}(x)$. For any $\a\in U_K$, we let  $u_\a(x) \in \GG$ be the unique power series such that $\reduce{u_\a} = [\a]_{\Fcal}.$  By setting $\m_\a(x) = [\a]_{\Fcal}(x)$, we now have two lifting   of $(\z,\m_\a).$ Namely, $([\b]_F, [\a]_F)$ and $(g, u_\a).$  Our goal is to show that for appropriate $\a\in U_K$,  we have $([\b]_F - g, [\a]_F - u_\a) \in B_r(\z, \m_\a)$ for all positive integer $r.$ From this, it's not difficult to deduce  the existence of an invertible power series  $\p(x)$ which gives the formal group $F^{\p}(x,y)$ such that $g(x)$ as well as all $u(x) \in \GG$ are endomorphisms of $F^{\p}(x,y).$

A key ingredient that we'll need is the following lemma whose proof will be postponed to the end of this section.
\begin{lemma}
\label{lem:same}
 Let $\alpha_1,\alpha_2,\beta\in \Ocal_K$ be such that  $0<v_K(\alpha_1-1)<v_K(\beta)\le v_K(\alpha_2-1)$.
  Suppose that we are given stable noninvertible power series $f(x)\in \Dcal_0(\Ocal_K)$ and invertible power series $,u_1(x),u_2(x)\in \Gcal_0(\Ocal)$ such that they commute with each other under the operation of substitution. Furthermore,  suppose that there is a positive integer $r$ such that $f\equiv [\beta]_F\pmod{\Mcal_K^r}$ and $u_i\equiv [\alpha_i]_F\pmod{\Mcal_K^r}$, for $i=1,2$. Then $f'(0)\equiv \beta\pmod{\Mcal_K^{r+1}}$ and  $u_2'(0)\equiv\alpha_2\pmod{\Mcal_K^{r+1}}$.
\end{lemma}

We apply~Lemma~\ref{lem:same}  to our situation.  Let $v_K(\beta)=s = v_K(g'(0))\ge 1$. Since the case $s=1$ has already been treated, we'll assume $s\ge 2$ in the following. Put $\alpha_1=1+p^{s-1}$ and $\alpha_2=1+p^s$. Let $u_1(x),u_2(x)\in \GG$ be such that  $\reduce{u_1}(x)=\m_1(x) = [\alpha_1]_\Fcal(x)$ and $\reduce{u_2}(x)=\m_2(x) = [\alpha_2]_\Fcal(x).$ We have the following.

\begin{claim}
For every positive integer $r$, there exists a power series $\ps_r\in \Gcal_0(\Ocal_K)$ with $\ps_r(x) \equiv x\pmod{\Mcal_K^r}$ such  that
\begin{align*}
\psi_r\circ g \circ\psi_r^{-1} & \equiv [\b]_F \pmod{\Mcal_K^{r+1}}, \\
 \psi_r\circ u_i \circ \psi_r^{-1} & \equiv [\a_i]_F \pmod{\Mcal_K^{r+1}},\; i = 1, 2.
\end{align*}
\end{claim}
We prove the claim by induction on $r.$
First we notice that $\a_1, \a_2$ and $\b$ as well as $u_1(x), u_2(x)$ and $g(x)$ satisfy the conditions in Lemma~\ref{lem:same} for $r =1$. Hence we conclude that $g'(0)\equiv \beta\pmod{\Mcal_K^{2}}$ and $u_2'(0)\equiv\alpha_2\pmod{\Mcal_K^{2}}$. Recall that $K$ is unramified over $\QQ_p$ of degree $f.$ Therefore it contains the group of $(q-1)$-roots of unity.
Let $\omega \in U_K$ be a fixed primitive $(q-1)$-th root of unity and let $u_\om(x)\in \GG$ be the unique invertible power series such that $u_\om'(0) = \om$.  Hence $u_\om$ is a generator for the cyclic subgroup of order $q-1$ in $\GG.$  Observe that  $\reduce{u_\om}$ and $[\om]_{\Fcal}$ both satisfy $\reduce{u_\om}'(0) = \reduce{\om} = [\om]_{\Fcal}'(0).$ Since the reduction map $\GG \to G \subset \Aut(\Fcal(x,y))$ is injective, we must have $\reduce{u_\om} = [\om]_{\Fcal}.$

Put $u(x) = u_2\circ u_\om$ and $\m = [\om\a_2]_{\Fcal}. $ By construction we have
\[\begin{array}{rlrl}
   \reduce{g}(x) &  =\zeta(x), & \quad  \reduce{u}(x) & =\mu(x)\quad  \text{and}  \\
   g'(0) & \equiv \beta \pmod{\Mcal_K^2} , & \quad  u'(0) & \equiv \omega\alpha_2\pmod{\Mcal_K^2}.
\end{array}\]
 In other words, $(g-[\beta]_F,u-[\omega\alpha_2]_F)=(d,w)\in Z_1^o(\zeta,\mu)$. On the other hand,  $\Ocal_K = \ZZ_p[\om\a_2]$ and $\b \in \ZZ_p[\om\a_2]$ since $K$ is unramified over $\QQ_p$ of degree $f$.
By~\cite[Corollary 9]{li2} and  the fact that the formal $\Ocal_K$-module $\left(\Fcal(x,y), \r_\Fcal\right)$ is of $\Ocal_K$-height one, we deduce that   $\dim_k(Z_1^o(\zeta,\mu)/B_1(\zeta,\mu))=0$. Hence, $(g-[\beta]_F,u-[\omega\alpha_2]_F)\in B_1(\z,\m).$ Thus  there exists  $\phi_1(x)\in\Gcal_0(\Ocal_K)$ with  $\p_1(x) \equiv x \pmod{\Mcal_K}$ such that
\[
 \phi_1\circ g\circ\phi_1^{-1}\equiv [\beta]_F\;\text{ and } \; \phi_1\circ u \circ\phi_1^{-1}\equiv  [\omega\alpha_2]_F\pmod{\Mcal_K^2}.
 \]

In particular, $\left(g-[\beta]_F\right) (x)$ is of the form $-\th_1(\z(x))$ for some $\th_1(x)\in x^2\series{M_1}{x}$ with $\th_1(x)\equiv \p_1(x) - x \pmod{\Mcal_K^2}.$ Now $\left( g - [\b]_F, u_i - [\a_i]_F \right) \in Z_1(\z, \m_i)$ by construction and  $\left(g-[\beta]_F\right)(x) =-\th_1(\z(x))$, we have  $\left([\b]_F - g, [\a_i]_F - u_i\right) \in B_1(\z, \m_i)$ for $i=1, 2.$ Since $\th_1(x)\equiv \p_1(x) - x \pmod{\Mcal_K^2}$, we conclude that the same invertible power series $\phi_1(x)$ gives
\begin{align*}
\phi_1\circ u_i \circ\phi_1^{-1}& \equiv [\alpha_i]_F\pmod{\Mcal_K^2},\; i =1 , 2\; \text{and} \\
 \phi_1\circ u_\om \circ\phi_1^{-1} & \equiv [\om]_F\pmod{\Mcal_K^2}.
\end{align*}
Therefore the claim is true for $r =1$ by setting $\psi_1 = \p_1.$

Let $r \ge 1$ and assume that there exists an invertible power series $\psi_r(x)\in \Gcal_0(\Ocal_K)$ satisfying $\psi_r(x)\equiv x\pmod{\Mcal_K^{r}}$ such that
\begin{align*}
\psi_r\circ g\circ\psi_r^{-1}& \equiv [\beta]_F \pmod{\Mcal_K^{r+1}},\;\text{and} \\
\psi_r\circ u_i \circ\psi_r^{-1}& \equiv [\alpha_i]_F\pmod{\Mcal_K^{r+1}},\;  i =1 , 2.
\end{align*}
Put $g_r = \psi_r\circ g\circ\psi_r^{-1}, u_{i,r} = \psi_r\circ u_i \circ\psi_r^{-1}, i=1, 2$ and $ u_{\om,r} = \psi_r\circ u_\om \circ\psi_r^{-1}$. Notice that we also have
$u_{\om,r}'(0)= \om $  and $\reduce{u_{\om,r}} = [\om]_{\Fcal}.$ Similarly, we see that $\a_1, \a_2$ and $\b$ as well as $u_{1,r}, u_{2,r}$ and $g_r$ satisfy the conditions in Lemma~\ref{lem:same} for $r+1$. Hence,
\[
g_r'(0)\equiv \b \pmod{\Mcal_K^{r+2}} \quad \text{and} \quad u_{2,r}'(0) \equiv \a_2\pmod{\Mcal_K^{r+2}}.
\]

Let $u_r = u_{2,r}\circ u_{\om,r}$ then $u_r$ satisfies  $u_r'(0) \equiv \om \a_2 \pmod{\Mcal_K^{r+2}}$ and $\reduce{u_r} = \m.$ Thus,
$(g_r - [\b]_F, u_r - [\a_2\om]_F)\in Z_{r+1}^o(\z,\m).$ The same reasoning as in the case for $r=1$, we also have $\dim_k(Z_{r+1}^o(\zeta,\mu)/B_{r+1}(\zeta,\mu))=0$ and hence  there exists an invertible power series  $\p_{r+1} \in \Gcal_0(\Mcal_K)$ with $\p_{r+1}(x)\equiv x\pmod{\Mcal_K^{r+1}}$ such  that
\begin{align*}
\p_{r+1}\circ g_r\circ\p_{r+1}^{-1}& \equiv [\beta]_F \pmod{\Mcal_K^{r+2}},\;\text{and} \\
\p_{r+1}\circ u_r \circ\p_{r+1}^{-1}& \equiv [\alpha_2\om]_F\pmod{\Mcal_K^{r+2}}.
\end{align*}
This also leads to
\[
\p_{r+1}\circ u_{i,r} \circ\p_{r+1}^{-1} \equiv [\alpha_i]_F\pmod{\Mcal_K^{r+2}}
\]
since by induction hypothesis  $(g_r - [\b]_F, u_{i,r} - [\a_i]_F)\in Z_{r+1}(\z, \m_i)$ for $i=1, 2$. Thus   $\psi_{r+1} = \p_{r+1}\circ \psi_r$  gives the desired invertible power series for the case of $r+1$ and completes the inductive proof for the claim.

The claim shows that there exists a sequence of invertible power series  $\{\psi_r\; : \; \psi_r\in \Dcal_0(\Ocal_K), r \ge 1\}$ such that $\psi_{r+1}(x)\equiv \psi_r(x)\pmod{\Mcal_K^r}$ for all $r\ge 1.$ Therefore $(\psi_r)_{r\ge 1}$ converges to an invertible power series $\p\in \Gcal_0(\Ocal_K)$ coefficientwise  such that $\reduce{\p}(x) = x$ and $\p\circ g \circ \p^{-1} = [\b]_F$ as desired. The proof for   Theorem~\ref{thm:lubin's conjecture}  will be completed provided that Lemma~\ref{lem:same} is proved.

\begin{proof}[Proof of Lemma~\ref{lem:same}]
By the theory of Lubin-Tate formal groups, we know that our formal group $F(x,y)$ is associated to a uniformizer  $\Pi$ of $K.$ Namely, we have
\[
[\Pi]_F(x)\equiv \Pi x \pmod{\text{degree 2}}, \quad [\Pi]_F(x) \equiv x^q \pmod{\Mcal_K}.
\]
On the other hand, we observe that the truth of the lemma remains unchanged under the conjugation by any invertible power series $\p\in \Gcal_0(\Ocal_K).$ Also, by~\cite[Theorem~1]{lubintate} all Lubin-Tate formal group associated to $\Pi$ are isomorphic over $\Ocal_K.$ Thus, we may prove the lemma under the assumption that our Lubin-Tate formal group $F(x,y)$ is the one such that  $[\Pi]_F(x) = \Pi x + x^q.$ We remark that under this assumption, by Lazard's comparison lemma \[F(x,y)\equiv x+y+\frac{1}{\Pi-\Pi^q}((x+y)^q-x^q-y^q)\pmod{\mathrm{degree}\,\,q+2}.\]

Let $\g \in \Mcal_K$ with $v_K(\g) = t \ge 1,$ then the endomorphism $[\g]_\Fcal$ of the formal group $\Fcal(x,y)$ is of the form $[\g]_{\Fcal}(x) = \hat{\l}(x^{q^t})$ where  $\hat{\l}$ is an automorphism of $\Fcal(x,y)$ as can be seen from~\cite[Theorem 6.3]{Lub}. In particular, $[\g]_{\Fcal}(x) = c_\g x^{q^t} + h(q^{t})$ with $\ord_x h(x) \ge 2$ and $c_\g\in \FF_q$ is a nonzero constant. Let's write $\a_i = 1 + \g_i,$ with $t_i = v_K(\g_i)$ for $ i=1, 2.$ By assumption $0 < t_1 < s \le t_2,$ where $s = v_K(g'(0)) = v_K(\b).$ From the remark above we have
\[
[\a_i]_\Fcal (x) = \Fcal(x, [\g_i]_\Fcal(x)) = x + c_i x^{q^{t_i}} \pmod{x^{q^{t_i} + 2}}, \; i =1 ,2.
\]
For $i=1, 2$ we set $\m_i = [\a_i]_{\Fcal}$ and  let $(d, w_i)\in \series{M_r}{x}\oplus \series{M_r}{x}$ be defined as follows
\begin{align*}
 d(x) & \equiv g(x) - [\b]_F(x) \pmod{\Mcal_K^{r+1}} \quad\text{and}  \\
 w_i(x) & \equiv u_i(x) - [\a_i]_F(x) \pmod{\Mcal_K^{r+1}}.
\end{align*}
By assumption, $(d, w_i) \in Z_r(\z, \m_i)$ where $\z = [\b]_\Fcal.$ Thus, $w_i$ satisfies the following functional equations.
\begin{equation}
\label{eq:functional}
d(\m_i(x)) = d(x)\cdot \m_i'(\z(x))+ w_i(\z(x)),\; i =1 ,2.
\end{equation}

Reducing~\eqref{eq:functional} modulo $x^{q^{t_i}+1}$ and using the fact that $\z(x) =\hat{\z}(x^{q^s})$ where $\hat{\z}(x) = c_3 x + O(x^2),\; c_3 \ne 0,$ we obtain the following
\begin{equation}
\label{eq:functional2}
  d(x+c_ix^{q^{t_i}})\equiv d(x)+ w_2(c_3x^{q^{s}} + O(x^{q^{2s}}))\pmod{x^{q^{t_i}+1}}, \; i =1 , 2.
\end{equation}
Since $t_1 < s$,  the above equation becomes $d(x+c_1x^{q^{t_1}})\equiv d(x) \pmod{x^{q^{t_1}+1}}$ for $i=1$. Then we must have $d'(0) = 0 $ in $M_r$  for otherwise the leading term of $d(x+c_1x^{q^{t_1}})- d(x)$ is of degree $q^{t_1}$ which is impossible. On the other hand, the assumption that $s\le t_2$  together with the fact that $d(x)$ does not have linear term, Equation~\eqref{eq:functional2} for $i=2$ shows that $w_2(x)$ can not have linear term either. Equivalently, $w_2'(0) = 0$ in $M_r.$ We have shown
\[
g'(0) \equiv \b \pmod{\Mcal_K^{r+1}} \;\;\text{and}\;\; u_2'(0) \equiv \a_2 \pmod{\Mcal_K^{r+1}}
\]
which are the desired congruences and Lemma~\ref{lem:same} is proved.
\end{proof}

\subsection{Final remarks}
\label{subsec:final remarks}
Let $g \in \Dcal_0(\Ocal_K)$ be a noninvertible power series and let $\Lambda_n(g)$ be the set of roots of its $n$-th iterate. That is,
\[
\Lambda_n(g) = \{\l\in \Kbar \mid g^{\circ n}(\l) = 0 \}.
\]
Let $K_{g,n} = K(\Lambda_n(g))$ be the Galois extension generated by $\Lambda_n(g)$ over $K.$ It is a natural question about the characterization of Galois groups for $K_{g,n}$ over $K.$ For arbitrary stable power series, it's not reasonable to expect a good answer to this question. However, if $g$ is an endomorphism of a formal group over $\Ocal_K$ then from the theory of formal groups a good description of the Galois group for $K_{g,n}$ over $K$ can be obtained.
As a direct consequence of Theorem~\ref{thm:lubin's conjecture}, we have the following result.
\begin{corollary}
Let $g\in \Dcal_0(\Ocal_K)$ be a stable noninvertible power series of height $d = [K :\QQ_p]$ and  satisfies either of  condition~(1) or~(2) in Theorem~\ref{thm:lubin's conjecture}.  Let $K_{g,n}$ be defined as above.  Then $K_{g,n}$ is an abelian extension over $K$ with Galois group isomorphic to $ \left(\Ocal_K/(g'(0)^n) \right)^{\ast}.$
\end{corollary}

 Let $\Dcal \subset \Dcal_0(\Ocal_K)$ be a family of commuting power series. Suppose that $\partial_0 : \Dcal \to \Ocal_K$ is surjective. By~\cite[Lemma~3.1]{li-proceedings}, any two commutative noninvertible power series have the same height. Therefore, the heights of noninvertible power series in $\Dcal$ are the same. Define the height of $\Dcal$ to be the height of any noninvertible power series in $\Dcal.$ Following~\cite[Definition~4.3.1]{lubin-64},  $\Dcal$ is called {\em full}~\cite[Remark~3.1]{sarkis10} if the map $\partial_0 : \Dcal \to \Ocal_K$ is surjective and the height of $\Dcal$ is equal to $[K : \QQ_p].$ If $\Dcal$ is full then  Sarkis conjectures that $\Dcal = \End_{\Ocal_K}(F)$ for some formal group $F(x,y)$ over $\Ocal_K.$ As another application of Theorem~\ref{thm:lubin's conjecture}, we give a proof of his conjecture.

 \begin{corollary}
Suppose that  $\Dcal \subset \Dcal_0(\Ocal_K)$ is full, then $\Dcal = \End_{\Ocal_K}(F)$ for some Lubin-Tate formal group over $\Ocal_K.$
 \end{corollary}

 \begin{proof}
 Since $\Dcal$ is full, there exists a noninvertible power series $g \in \Dcal$ with $g'(0) = \partial_0(g) = \pi$ and $\hgt(g) = [K :\QQ_p].$   Notice that $\Stab_{\Ocal_K}(g) = \partial_0^{-1}(U_K).$ It's clear that $g$ satisfies condition~(1) in Theorem~\ref{thm:lubin's conjecture}. Therefore, $g\in \End_{\Ocal_K}(F)$ and $\Stab_{\Ocal_K}(g) = \Aut_{\Ocal_K}(F)$
 for some Lubin-Tate formal group over $\Ocal_K.$ From this, we have $\Dcal\subset \End_{\Ocal_K}(F).$

 As $F(x,y)$ is a Lubin-Tate formal group over $\Ocal_K,$ its endomorphism ring is isomorphic to $\Ocal_K$ with isomorphism $\partial_0 : \End_{\Ocal_K}(F) \to \Ocal_K.$ Since the restriction of the map $\partial_0$ to $\Dcal$ is surjective, we must have $\Dcal = \End_{\Ocal_K}(F)$ as asserted.
 \end{proof}

It is natural to ask whether or not conditions~(1) and (2) in Theorem~\ref{thm:lubin's conjecture} are necessary.
We believe that the conclusion in Theorem~\ref{thm:lubin's conjecture} is still true without these two conditions. We state the version of Theorem~\ref{thm:lubin's conjecture} without condition~(1) and (2), which strengthens Sarkis' conjecture on full set of commuting formal power series.

\begin{conjecture}
Let $g\in \Dcal_0(\Ocal_K)$ be a stable noninvertible power series over $\Ocal_K$ with height equal to $[K :\QQ_p]$ such  that  $\partial_0\left(\Stab_{\Ocal_K}(g)\right) = U_K.$ Then, $g$ is an endomorphism for some Lubin-Tate formal group $F(x,y)$ over $\Ocal_K$ and $\Stab_{\Ocal_K}(g) = \Aut_{\Ocal_K}(F).$
\end{conjecture}


\end{document}